\documentclass[10pt,a4paper,reqno]{amsart}
\usepackage{amssymb,amsmath,amsthm,amstext,amsfonts,latexsym}
\usepackage[dvips]{graphicx}
\usepackage{color}
\usepackage{epstopdf}
\usepackage{pictexwd,dcpic}

\makeatletter \@addtoreset{equation}{section} \makeatother

\theoremstyle{plain}

\theoremstyle{plain}
\newtheorem{maintheorem}{Theorem}

\newtheorem{theorem}{Theorem }[section]
\newtheorem{proposition}[theorem]{Proposition}
\newtheorem{lemma}[theorem]{Lemma}
\newtheorem{corollary}[theorem]{Corollary}
\theoremstyle{definition} \theoremstyle{remark}
\newtheorem{remark}[theorem]{Remark}
\newtheorem{definition}[theorem]{Definition}
\DeclareMathAlphabet{\mathpzc}{OT1}{pzc}{m}{it}

\newcommand{\field}[1]{\mathbb{#1}}
\newcommand{\RR}{\field{R}}

\newcommand{\ZZ}{\field{Z}}
\newcommand{\NN}{\field{N}}

\newcommand{\loc}{\textnormal{loc}}

\numberwithin{equation}{section}

\begin{document}
\large

\title[Strange attractors and wandering domains near a homoclinic cycle to a bifocus]{Strange attractors and wandering domains \\ near a homoclinic cycle to a bifocus}

\author[Alexandre A. P. Rodrigues]{Alexandre A. P. Rodrigues \\ Centro de Matem\'atica da Universidade do Porto \\ 
and Faculdade de Ci\^encias da Universidade do Porto\\
Rua do Campo Alegre 687, 4169--007 Porto, Portugal}
\address[A. A. P. Rodrigues]{Centro de Matem\'atica da Universidade do Porto\\
and Faculdade de Ci\^encias da Universidade do Porto\\ 
Rua do Campo Alegre 687, 4169--007 Porto, Portugal}
\email[A.A.P.Rodrigues]{alexandre.rodrigues@fc.up.pt}

\date{\today}

\begin{abstract}
In this paper, we explore the three-dimensional chaotic set near a homoclinic cycle to a hyperbolic bifocus at which the vector field has negative divergence. If the invariant manifolds of the bifocus satisfy a non-degeneracy condition, a sequence of hyperbolic suspended horseshoes arises near the cycle, with one expanding and two contracting directions. We extend previous results on the field and we show that the first return map to a given cross section may be approximated by a map exhibiting heteroclinic tangencies associated to two periodic orbits. When the cycle is broken, under an additional hypothesis about the coexistence of two heteroclinically related periodic points (one without dominated splitting into one-dimensional sub-bundles), the heteroclinic tangencies can be slightly modified in order to satisfy Tatjer's conditions for a generalized tangency of codimension two. This configuration may be seen as the organizing center, by which one can obtain Bogdanov-Takens bifurcations and therefore, strange  attractors, infinitely many sinks and non-trivial contracting wandering domains. 

 \end{abstract}

\maketitle
{\small\noindent\emph{MSC 2010:} 34C37, 37C29, 37G35, 37D45 }
\vspace{0.5cm}

\section{Introduction}
The homoclinic cycle to a bifocus provides one of the main examples of the occurrence of chaotic dynamics in four-dimensional vector fields. Examples of dynamical systems from applications where these homoclinic cycles play a basic role can be found in  \cite{Glendinning, GL}. Results in \cite{Barrientos} show that homoclinic orbits to bifoci arise generically in unfoldings of four-dimensional nilpotent singularities of codimension 4. Family (1.2) in \cite{Barrientos}  has been widely studied in the literature because of its relevance in many physical settings as, for instance, the study of travelling waves in the Korteweg-de Vries model.

The striking complexity of the dynamics near this type of homoclinic cycles has been discovered and investigated by Shilnikov \cite{Shilnikov67A, Shilnikov70}, who claimed the existence of a countable set of periodic solutions of saddle type. It was shown that, for any $N \in \mathbb{N}$ and for any local transverse section to the homoclinic cycle, there exists a compact invariant hyperbolic set on which the Poincar\'e map is topologically conjugate to the Bernoulli shift on $N$ symbols. A sketch of the proof has been presented in \cite{Wiggins}. In the works \cite{FS, LG}, the formation and bifurcations of periodic solutions were studied. Motivated by \cite{ALR, RLA}, the authors of \cite{IbRo} describe the hyperbolic suspended horseshoes that are contained in any small neighbourhood of a double homoclinic cycle to a bifocus and showed that switching and suspended horseshoes are strongly connected.

The spiralling geometry of the non-wandering set near the homoclinic cycle associated to the bifocus has been partially described in \cite{FS}, where the authors studied generic unfoldings of the cycle. Breaking the cycle, and using appropriate first return maps, the authors visualized the structure of the spiralling invariant set which exists near the cycle.  In the reversible setting, the authors of \cite{Hart, Rodrigues_AIMS} proved the existence of a family of non-trivial (non-hyperbolic) closed trajectories and subsidiary connections near this type of cycle. See also the works by Lerman's team \cite{KL2009, Lerman2000} who studied cycles to a bifocus in the hamiltonian context.  In the general context, the complete understanding of the structure of this spiralling set is a hard task.

An important open question related to the homoclinic cycle to a bifocus is what type of dynamics typically occurs.   
In this paper, we study the dynamics near a homoclinic cycle to a bifocus at which the vector field has negative divergence, so that the flow near the equilibrium contracts volume. We are particularly interested in the occurrence of \emph{strange attractors} and \emph{non-trivial wandering domains}. We show that these phenomena occur for small $C^1$-perturbations of the vector field which, in principle,  no longer have the original homoclinic cycle. 

\subsection{Strange attractors}
Many aspects contribute to the richness and complexity of a dynamical system. One of them is the existence of strange attractors. According to \cite{Homburg2002, Viana93}:

\begin{definition}
A \emph{(H\'enon-type) strange attractor} of a three-dimensional dissipative diffeomorphism $R$, defined in a compact and riemannian manifold, is a compact invariant set $\Lambda$ with the following properties:
\begin{itemize}
\item $\Lambda$  equals the closure of the unstable manifold of a hyperbolic periodic point;
\item the basin of attraction of $\Lambda$   contains an open set (and thus has positive Lebesgue measure);
\item there is a dense orbit in $\Lambda$ with a positive Lyapounov exponent (exponential growth of the derivative along its orbit);
\item $\Lambda$ is not hyperbolic.
\end{itemize}
A vector field possesses a (H\'enon-type) strange attractor if the first return map to a cross section does. In \cite{Viana97}, there is another definition of strange atractor contemplating \emph{two} expanding directions.  
\end{definition}

The rigorous proof of the strange character of an invariant set is a great challenge and the proof of the existence of such attractors in the unfolding of the homoclinic tangency is a very involving task  (as discussed in \cite{Barrientos}). 
Mora and Viana \cite{MV93} proved the emergence (and persistence) of strange attractors in the process of creation or destruction of the Smale horseshoes that appear through a bifurcation of a tangential homoclinic (sectionally dissipative) point. 

In the unfolding of a non-contracting Shilnikov cycle associated to a saddle-focus in $\RR^3$ (details in \cite{Shilnikov65}),  horseshoes appear and disappear by means of generic homoclinic bifurcations, leading to persistent non hyperbolic strange attractors like those described in  \cite{MV93}. These tangencies give rise to suspended H\'enon-like strange attractors. Without breaking the cycle, Homburg \cite{Homburg2002} proved the coexistence of strange attractors and attracting 2-periodic solutions near a homoclinic cycle to a saddle-focus in $\RR^3$, when moving the saddle-value (see also \cite{OS92}). He also proved that, despite the existence of strange attractors, a large proportion of points near the homoclinic cycle lies outside the basin of the attractor. Under a particular configuration of the spectrum of the vector field at the saddle (equal to 1), Pumari\~no and Rodr\'iguez \cite{PR2001} proved that infinitely many of these strange attractors can coexist in non generic families of vector fields with a Shilnikov cycle, for a positive Lebesgue measure set of parameters.

The homoclinic cycle to a bifocus in $\RR^4$ seems to be the scenario for more complicated dynamics than those inherent to the saddle-focus in $\RR^3$, where the existence of such strange attractors has been proved. As far as we know, no result has been established relating the existence of bifocal homoclinic bifurcations with the existence of (persistent) strange attractors. In Theorem \ref{main_thB}, we prove that the first return map to a given cross section of the a cycle associated to a bifocus can be $C^1$-approximated by another map exhibiting strange attractors. 


\subsection{Non-trivial wandering domains}
A wandering domain for a diffeomorphism is a non-empty connected open set whose forward orbit is a sequence of pairwise disjoint open sets.  More precisely:

\begin{definition}
A \emph{non-trivial wandering domain} (or just wandering domain) for a given map $R$ on a Riemannian manifold $M$ is a non-empty connected open set $D \subset M$ which satisfies the following conditions:
\medbreak
\begin{itemize}
\item $R^i(D)\cap R^j(D)=\emptyset$ for every $i,j\geq 0$ ($i\neq j$)
\item the union of the $\omega$-limit sets of points in $D$ for $R$, denoted by $\omega(D,R)$, is not equal to a single periodic orbit.
\medbreak
\end{itemize}
A wandering domain $D$ is called \emph{contracting} if the \emph{diameter} of $R^n(D)$ converges to zero as $n \rightarrow +\infty$.
\end{definition}

\medbreak

 In the early $20^{th}$ century, the authors of \cite{Bohl, Denjoy} constructed examples of $C^1$ diffeomorphisms on a circle which have contracting wandering domains where the union of the $\omega$--limit sets of points is a Cantor set. See also \cite{Martens} and references therein. Similar behaviours in different contexts may be found in  \cite{Bonatti94, Knill, Kwakkel, Norton}.
The existence of non-trivial wandering domains in nonhyperbolic dynamics has been studied by Colli and Vargas \cite{CV2001}, through  a countable number of perturbations on the gaps of an affine thick horseshoe with persistent tangencies. The conjecture about the existence of contracting wandering sets near Newhouse regions was recently proved in \cite{KS} for diffeomorphisms, and partially in \cite{LR2016} for flows, when the authors were exploring persistent historic behaviour realised by a set with positive Lebesgue measure. 

In the context of diffeomorphisms of the circle, if sufficient differentiability exists, Denjoy \cite{Denjoy} proved that non-trivial wandering domains could not exist. The absence of wandering domains is the key for the classification of one-dimensional unimodal and multimodal maps, in real analytic category, a subject which has been discussed in \cite{MvS, Lyubich, vS}.  
For rational maps of the Riemannian sphere, we address the reader to \cite{Martens}.  
Very recently,  Kiriki \emph{et al} \cite{KNS} presented a sufficient condition for three-dimensional diffeomorphisms having heterodimensional cycles (and thus non-transverse equidimensional cycles $C^1$-close) which can be  $C^1$-approximated by diffeomorphisms with non-trivial contracting wandering domains and strange attractors. 

 A natural question arises: is there a configuration for a flow  having a first return map with equidimensional cycles similar to those described in \S3 of \cite{KNS}? In other words, could we describe a general configuration (for a flow) giving a criterion for the existence of non-trivial wandering domains? Theorem \ref{main_thA} gives a partial answer to this question.
 
 \subsection{Structure of the paper}
 The goal of this paper is to show that a homoclinic cycle associated to a bifocus may be considered as a criterion for four-dimensional flows to be $C^1$-approximated by other flows exhibiting strange attractors and contracting non-trivial wandering domains. 
 
 The main results are stated and discussed in \S \ref{main results}, after collecting relevant notions in \S \ref{preliminaries}. Normal form techniques are used in \S \ref{Local_dynamics} to construct local and return maps. Section \ref{Local_Geom} deals with the geometrical structures which allow to get an understanding of the dynamics. After reviving the proof of the existence of hyperbolic horseshoes whose suspension accumulates on the cycle (see \S \ref{horseshoes}), owing the results of \cite{Broer96, KNS, OS92, Tatjer}, in \S \ref{perturbations} we $C^1$-approximate the first return map to the cycle by another diffeomorphism  exhibiting a Tatjer tangency. 
  This codimension-two bifurcation leads to Bogdanov-Takens bifurcations and subsidiary homoclinic connections associated to a sectionally dissipative saddle.  In \S \ref{final_proof}, we prove of Theorems \ref{main_thB} and \ref{main_thA}. The itinerary of their proof is summarised in Appendix \ref{app1}.

  The last step of the proof of Theorem \ref{main_thA} is similar to \cite{KNS}. For the sake of completeness, we revisit the proof, addressing the reader to the original paper where the proof has been done.
Throughout this paper, we have endeavoured to make a self contained exposition bringing together all topics related to the proofs. We have stated short results and we have drawn illustrative figures to make the paper easily readable.

\section{Preliminaries}

\label{preliminaries}
For $k\geq 5$ and $A$ a compact  and boundaryless subset of $\RR^4$, we consider a $C^k$ vector field $f: A \rightarrow \RR^4$ defining a differential equation:
 \begin{equation}
    \label{general1}
     \dot{x}=f(x)
 \end{equation}
and denote by $\varphi(t,x)$, with $t \in \RR$, the associated flow. In this section, we introduce some essential topics that will be used in the sequel. 

\subsection{$\omega$-limit set}
For a solution of (\ref{general1}) passing through $x\in \RR^4$, the set of its accumulation points as $t$ goes to $+\infty$ is the $\omega$-limit set of $x$ and will be denoted by $\omega(x)$. More formally, 
$$
\omega(x)=\bigcap_{T=0}^{+\infty} \overline{\left(\bigcup_{t>T}\varphi(t, x)\right)}.
$$ 
It is well known that $\omega(x)$  is closed and flow-invariant, and if the $\varphi$-trajectory of $x$ is contained in a compact set, then 
$\omega(x)$ is non-empty.

\subsection{Homoclinic cycle}
In this paper, we will be focused on an equilibrium $O$ of (\ref{general1}) such that its spectrum (\emph{i.e.} the eigenvalues of $df(O)$) consists of four non-real complex numbers whose real parts have different signs. It is what one calls a \emph{bifocus}. A \emph{homoclinic connection} associated to $O$ is a trajectory biasymptotic to $O$ in forward and backward times.

\subsection{Terminology}
In this subsection, we recall the terminology given in \cite{Tatjer} for diffeomorphisms. 
We begin with some definitions concerning fixed points of a diffeomorphism $R$ on a three-dimensional Riemannian manifold $M$, which may be considered as a compact subset of $\RR^3$.
\medbreak

Let $R : M \rightarrow M$ be a diffeomorphism,  $P \in M$ be a hyperbolic fixed point of $R$ (\emph{i.e.} $R(P) = P$) and denote by $\mu_1, \mu_2, \mu_3 \in \mathbb{C}$ the eigenvalues of $DR(P)$.

\begin{definition} Let $P$ be a fixed point of $R$. We say that:
\begin{itemize}
\item $P$ is \emph{dissipative} if the product of the absolute value of the eigenvalues of $DR(P)$ is less than 1 (\emph{i.e.} $|\mu_1 \mu_2 \mu_3|<1$).
\medbreak
\item $P$ is \emph{sectionally dissipative} if the product of the absolute value of any pair of eigenvalues of $DR(P)$ is less than 1 (\emph{i.e.} $|\mu_1 \mu_2 |<1$, $|\mu_2 \mu_3 |<1$  and $|\mu_1 \mu_3 |<1$  ).
\end{itemize}
\end{definition}
\medbreak
By the Stable Manifold Theorem \cite{PM}, given a
saddle fixed point $P$ (for the map $R$) there exist the stable and unstable invariant manifolds that we denote by $W^s(R, P)$, and $W^u (R, P)$ respectively, and are defined by
$$
W^s(R,P)=\left\{Q\in M: \lim_{n \rightarrow +\infty} R^n(Q)=P\right\} \quad \text{and} \quad W^u(R,P)=\left\{Q\in M: \lim_{n \rightarrow +\infty} R^{-n}(Q)=P\right\}.
$$
As usual, if it is not necessary, we shall not write explicitly the dependence of the invariant manifolds on the map $ R$.

\begin{definition}
The \emph{index of stability} of a hyperbolic fixed point is the dimension of its stable manifold.
\end{definition}

\medbreak
\begin{definition}[\cite{GGT2007, Tatjer}, adapted]
Suppose that $R$ is as above and $P$ is a saddle fixed point.
\begin{enumerate}
\medbreak
\item If  $|\mu_1| < |\mu_2|<1<|\mu_3|$, then:
\medbreak
\begin{enumerate}
\item  the \emph{strong stable manifold} of $P$, denoted by $W^{ss}(P)$,  lies on $W^s(P)$ and is tangent to the eigenspace associated to $\mu_1$ at $P$. The manifold $W^{ss}(P)$ is unique, one-dimensional and is as smooth as $R$;
\medbreak
\item the set $W^s(P)$ is foliated by leaves of a \emph{strong stable foliation} $\mathcal{F}^{ss}(P)$. Every leaf of $\mathcal{F}^{ss}(P)$ is transverse to the eigenspace associated to $\mu_2$ and  $W^{ss}_{\loc}(P)$ is one of the leaves;
\medbreak
\item the \emph{center-unstable manifold} of $P$, $W^{cu}(P)$, is an invariant manifold containing $W^u(P)$ and touching the  invariant linear subspace of $T_P M$ associated to the eigenvalues $\mu_2$ and $\mu_3$, at $P$. This manifold is (in general) $C^1$-smooth and it is not unique.
\end{enumerate}
\medbreak

\item If  $|\mu_1| <1< |\mu_2|<|\mu_3|$, the \emph{strong unstable foliation} of $P$, denoted by   $\mathcal{F}^{uu}(P)$ is the strong stable foliation of $P$ with respect $R^{-1}$ and the \emph{center-stable invariant manifold} of $P$ is the center-unstable manifold  of $P$ with respect to $R^{-1}$.
\end{enumerate}
\end{definition}

 More details about foliations and tangent bundles may be found in \cite{GGT2007, Homburg_book}. We now introduce the concept of signature (adapted to our purposes) of a periodic point adapted from \cite{Bonatti94}.
\begin{definition}
\label{signature}
Let $R$ be a diffeomorphism as above and let $P$ be a hyperbolic periodic point of period $\xi\geq 1$ with $\dim W^u(P)=2$. Given the finest $DR^{\xi}(P)$--invariant dominated splitting $E_p^u= E_1^u \oplus E_2^u$, we define the \emph{unstable signature} of $P$ to be the pair $(\dim  E_1^u, \dim E_2^u)$.  
\end{definition}


\subsection{Generalized homoclinic tangency}

Let $R$ be a diffeomorphism on $M$ which has a homoclinic tangency of a fixed point $P$. 
Suppose the derivative for $R$ at $P$ has real eigenvalues $\mu_1$, $\mu_2$ and $\mu_3$ satisfying $|\mu_1| < |\mu_2|<1<|\mu_3|$. In addition, assume that there are $C^1$ linearizing local coordinates $(x, y, z)$ for $R$  on a  small neighbourhood $U$ of $P$ (see Remark \ref{Gonchenko generalisation} below) such that:
$$
P = (0,0,0)\qquad \text{and} \qquad  R(x,y,z) = (\mu_1 x \, , \mu_2 y\, , \mu_3 z)
$$
for any $(x,y,z) \in U$ -- see Figure \ref{generalized_tangency}(a). In $U$, the local stable and unstable manifolds of $P$ are given respectively as:
$$
W^s_{loc} (P) = \{(x,y,0) : \quad |x|,|y| < \varepsilon\}, \qquad  W^u_{loc} (P)  = \{(0,0,z) ; \quad  |z| < \varepsilon \}
$$
for some small $\varepsilon > 0$. Moreover, as illustrated in Figure  \ref{generalized_tangency}(a),  one has the local strong stable $C^1$ foliation $\mathcal{F}^{ss}(P)$ 
in $W^s (P)$ such that, for any point $x_0 = (x^\star,y^\star, 0) \in W^s (P)$, the leaf $\ell ^{ss}(x_0)$ of $\mathcal{F}^{ss}(P)$ containing $x_0$ is given as:
$$
\ell ^{ss}(x_0) = \{(x, y^\star, 0) : \quad |x-x^\star| < \varepsilon \}.
$$

\begin{figure}[ht]
\begin{center}
\includegraphics[height=5.0cm]{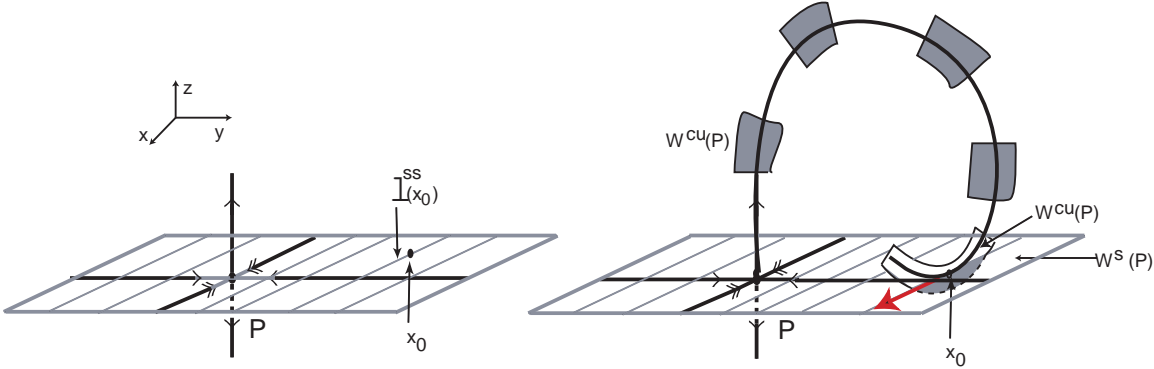}
\end{center}
\caption{\small Generalized tangency: (a) local coordinates near $P$ and (b) global geometry near the homoclinic orbit described in \cite{Tatjer}. The centre-unstable bundle ($DR$-eigenspace associated to the eigenvalues $\mu_2$ and $\mu_3$) at $P$ is extended along  $W^u(P)$.}
\label{generalized_tangency}
\end{figure}

\begin{remark}
The linearisation assumption is a \emph{Baire-generic} assumption for families of diffeomorphisms having saddle fixed points. Gonchenko \emph{et al} \cite{GGT2007} generalized the results of \cite{Tatjer} without this assumption. 
\label{Gonchenko generalisation}
\end{remark}

Suppose that the invariant manifolds of $P$ have a quadratic tangency at $x_0$. 
We introduce the definition of a new type of codimension two homoclinic bifurcation, which may be seen as a collision of a quadratic homoclinic tangency and a generalized homoclinic transversality (see \S 2.3 of \cite{Tatjer}).
\begin{definition}
\label{Tatjer_condition}
We say that a homoclinic tangency to $P$ satisfies the \emph{Tatjer condition} (type I of Case A of \cite{Tatjer}) if the following conditions hold (see Figure \ref{generalized_tangency}(b)):
\begin{itemize}
\medbreak
\item[\textbf{[T1]:}] the point $P$ is dissipative but not sectionally dissipative for $R$.
\medbreak
\item[\textbf{[T2]:}] the manifolds $W^u(P)$ and $W^s(P)$ have a quadratic tangency at $x_0$ which does not belong to the strong stable manifold of $P$, $W^{ss}(P)$.
\medbreak
\item[\textbf{[T3]:}] the manifold $W^u(P)$ is tangent to the leaf $\ell^{ss}(x_0)$ of $\mathcal{F}^{ss}(P)$ at $x_0$.
\end{itemize}
\end{definition}

\medbreak

\begin{itemize}
\item[\textbf{[T4]:}] the center-unstable manifold of $P$ is transverse to the surface defined by $W^s(P)$ at $x_0$.
\end{itemize}

\begin{remark}
For the quadratic tangency point $x_0\in M$, we consider the forward image $\overline{x}_0 = f^{-n_0} ({x}_0)$ for a large $n_0 \geq 0$. Let $U(\overline{x}_0)$ be the plane containing
$\overline{x}_0$ and such that $T_{\overline{x}_0} U(\overline{x}_0)$ is generated by $$\left(\frac{\partial}{\partial y}\right)\Big|_{\overline{x}_0},\left( \frac{\partial}{\partial z} \right)\Big|_{\overline{x}_0} \in T_{\overline{x}_0} M$$ (in local coordinates of $P$). 
Figure 5 of \cite{KNS} and Figure \ref{generalized_tangency}(b) of the present paper illustrate the positions of $U(\overline{x}_0)$ and $W^{s}(P)$ in different parts of the phase space. 
Using the terminology of \cite{KNS}, condition \textbf{[T4]} may be stated as: \emph{the sets $U(\overline{x}_0)$ and $W^s_\loc(P)$ are transverse at $\overline{x}_0$}. This concept is valid if we replace a fixed point of $R$ by a periodic orbit of $R$. 

\end{remark}

\bigbreak

\subsection{Denjoy surgery}
\label{Denjoy revision}
In this subsection, we review the Denjoy construction \cite{Denjoy} to obtain wandering domains on the circle. For $\omega \in \RR \backslash \mathbb{Q}$, define the map $\tau_\omega$ on $S^1= \RR \pmod{2\pi}$ as 
$$
\tau_\omega (\theta) = \theta + 2\pi \omega  \qquad \text{where }\qquad \theta \in S^1.
$$
Now take a point $\theta_0\in S^1$. Then, for each $n \in \NN$, we remove from $S^1$ the point  $\tau_\omega^n(\theta_0)$ and we replace it by a small enough interval $I_n$ satisfying the following properties:
\bigbreak
\begin{itemize}
\item for each $n\in \NN$, $\ell(I_n)>0$, where $\ell$ denotes the the 1-dimensional Lebesgue measure;
\bigbreak
\item $\sum_{j=0}^\infty \ell(I_n)<+\infty$.
\end{itemize}
\bigbreak
The result of this surgery is still a simple closed curve. For each $n\in \NN$,  extend the map $\tau_\omega$ by choosing a orientation-preserving diffeomorphism $h_n:I_n \rightarrow I_{n+1}$. It is easy to see that this extends $\tau_\omega$ to a homeomorphism of the new closed curve with no periodic points. Denjoy \cite{Denjoy} proved that the rotation number of the new map is irrational and no point in the interval $I_n$ ever returns to $I_n$. This is an example of a wandering domain for a map of the circle.  This construction cannot be performed in the $C^2$ category.

\section{Main results: overview}
\label{main results}

\subsection{Description of the problem}
\label{description}
The object of our study is the dynamics around a
homoclinic cycle $\Gamma$ to a bifocus defined on $\RR^4$ for which we give a rigorous description here. Let $\mathcal{X}^5(\RR^4)$ the Banach space of $C^5$ vector fields on $\RR^4$ endowed with the $C^5$-Whitney topology. Our object of study is a one-parameter family of $C^5$ vector fields $f_\lambda:\RR^4 \rightarrow \RR^4$ with a flow given by the unique solution $x(t)=\varphi(t,x) \in \RR^4$ of
\begin{equation}
\label{general}
\dot{x}=f_\lambda(x) \qquad x(0)=x_0\in \RR^4 \qquad \lambda \in \RR
\end{equation}
satisfying the following hypotheses for $\lambda=0$:
\medbreak
\begin{description}
 \item[\textbf{(P1)}]\label{P1}  The point $O=(0,0,0,0)$ is an equilibrium point.
 \medbreak
 \item[\textbf{(P2)}]\label{P2} The spectrum of $df_0(O)$ is $\{-\alpha_1 \pm i\omega_1,\alpha_2 \pm i\omega_2\}$ where $0<\alpha_2 < \alpha_1$ and $\omega_1, \omega_2>0$.
 \medbreak

 \item[\textbf{(P3)}]\label{P3} There is (at least) one trajectory $\gamma$ biasymptotic to $O$. The homoclinic cycle is given by $\Gamma=\{O\}\cup \gamma$.
 \medbreak
 \item[\textbf{(P4)}]\label{P4} For all $t \in \RR$, one has $\dim \left(T_{\gamma(t)}W^u(O) \cap T_{\gamma(t)}W^s(O)\right)=1$.
   \end{description}
 \medbreak
 In addition, we state the following \emph{non-degeneracy condition}:  
 \begin{description}
 \medbreak
 \item[\textbf{(P5)}]\label{P5}   For $ \lambda >0$ small, the cycle $\Gamma$ is broken in a generic way. 
  \medbreak
  \end{description}

\begin{remark}
Property (P4) is  equivalent to:
$$ \forall t \in \RR, \qquad codim \left(span\{T_{\gamma(t)}W^u(O), T_{\gamma(t)}W^s(O)\}\right)=1.$$
This property defines an open and dense condition in the $C^r$-topology, $r\geq 2$. 
\end{remark}

 Throughout the present  paper we confine ourselves to $\RR^4$ which may seem restrictive. A reduction from a higher-dimensional system to the four-dimensional case can be achieved by a \emph{center manifold reduction} near the cycle  \cite{Sandstede}.

\bigbreak
Let $\mathcal{T}$ be a small tubular neighbourhood of $\Gamma$ and let $\Sigma$ be a transverse section which cuts $\gamma$ at a unique point $q\in \Gamma$. For $\lambda=0$,  it has been proved in \cite{IbRo} that there exists an invertible first return map $R_{0}: S\rightarrow \Sigma$ defined on a set $S  \subset \Sigma $  whose closure contains $\{q\}$.  The next result summarises what is known about the dynamics of (\ref{general}) inside $\mathcal{T}$ (adapted to the case $\alpha_2<\alpha_1$).
\begin{theorem}[\cite{IbRo, Shilnikov70}, adapted]
\label{Th2b}
Under hypotheses (P1)--(P4) on the differential equation \eqref{general}, for any tubular neighbourhood $\mathcal{T}$ of the cycle $\Gamma$ and every cross-section to the flow $\Sigma\subset \mathcal{T}$ at a point  $q\in \Gamma$, there exist a set of initial conditions $S\subset \Sigma$, a $C^r$ return map $R_0: S\rightarrow \Sigma$ with $r\geq 5$, and a set of initial conditions $\Lambda \subset S$ such that:

\medbreak

\begin{enumerate}
\item[(a)]   The first return map to  $\Lambda$ has a countable family of uniformly hyperbolic compact invariant sets $\left(\Lambda_{M}\right)_{M \in \mathbb{N}: \,M \geq M_0}$ in each of which the dynamics is conjugate to a full shift over a finite number of symbols. 
\medbreak
\item[(b)] The set  $\Lambda :=\bigcup_{M=M_0}^\infty \Lambda_{M}$ accumulates in $\Gamma$ and the number of symbols coding the first return map to $S$ tends to infinity as we approach the cycle $\Gamma$.
\medbreak
\item[(c)] The map $R_{0}|_\Lambda$ induces on the tangent bundle $T\Lambda$ one expanding and two contracting directions.
\end{enumerate}
\end{theorem}

In \S \ref{horseshoes}, we review the main steps of the proof of Theorem \ref{Th2b}(a--c) and we clarify the role of both the stable and unstable manifolds of the periodic solutions of $(\Lambda_M)_{M\in \NN}$, in the overall structure of the maximal invariant set in $\mathcal{T}$.  We reconstruct the proof of the existence of horseshoes $(\Lambda_M)_{M \in \mathbb{N}}$ in order to understand the geometry of the problem and the dynamics emerging when the cycle $\Gamma$ is broken (used in \S 7).  
In the spirit of \cite{Rodrigues2_2013},  the proof of Theorem \ref{Th2b} allows us to conclude that the hyperbolic part of the shift dynamics does not trap most solutions in the neighbourhood of $\Gamma$. 
\medbreak
For $\lambda \approx 0$, let $R_{\lambda}: S_1\rightarrow \Sigma$ be the return map associated to $S_1\subset \Sigma$, after the addition of a perturbing term that breaks $\Gamma$ -- see (P5).  Of course, among the infinity of horseshoes that occur for $\lambda=0$, a finite number (but arbitrarily large) of them persists under small  smooth perturbations. 

\medbreak

\subsection{Main results and strategy}
For $\lambda=0$, the union of horseshoes accumulates on $\Gamma$ as well as the invariant manifolds of its periodic orbits. Under a technical condition, this implies that there are diffeomorphisms arbitrarily close to $R_{0}$ for which we can find a heteroclinic tangency associated to hyperbolic periodic points of $\Lambda_N, \Lambda_M$, for large $N,M\in \NN$, with different signatures. A small local perturbation may be performed and the previous configuration  may be approximated by another satisfying the condition described in Definition \ref{Tatjer_condition}, ensuring important dynamical properties nearby. One of them is our first main result: 

\begin{maintheorem}
\label{main_thB}
If $f$ satisfies hypotheses (P1)--(P5) and (TH), $\Sigma$ is a sufficiently small cross section to $\gamma$ and $R_0$ is the Poincar\'e map associated  to a subset of $\Sigma$, then there exists a diffeomorphism $G_A$ arbitrarily $C^1$-close to $R_0$, exhibiting (H\'enon-type) strange attractors  and/or infinitely many sinks. 
 \end{maintheorem}

The proof of Theorem \ref{main_thB} may be found in \S \ref{8.1}, where we specify the technical hypothesis (TH). The strange attractors found in Theorem~\ref{main_thB} have one positive Lyapunov exponent.  It is  the closure of an invariant unstable manifold of a hyperbolic periodic orbit, and thus its shape might be very complicated. Tatjer's strategy allows us to conclude that they are contained in the center manifold associated to the Takens-Bogdanov bifurcation.  Using \cite{Barrientos}, one immediate consequence of Theorem \ref{main_thB} is:
\begin{corollary}
A vector field in a generic unfolding of a four-dimensional nilpotent singularity of codimension four may be $C^1$-approximated by a vector field containing a (H\'enon-type) strange attractor.
\end{corollary}

The authors of \cite{BIR2016} proved that suspended robust heterodimensional cycles can be found arbitrarily close to any non-degenerate bifocal homoclinic orbit of a Hamiltonian vector field. Perturbing the cycle in a special manner,  it is possible to obtain a sectionally dissipative homoclinic point and the existence of strange attractors near the original vector field may also be obtained.
\medbreak

Starting with a Tatjer homoclinic tangency, it is possible to find a two-parameter family of diffeomorphisms $ G_{(a,b)}$ and a sequence of parameters $(a_n, b_n)$ converging to $(0,0)$ for which for large $n\in \NN$, the diffeomorphism $G_{(a_n,b_n)}$ has a $n$-periodic smooth {attracting} circle generated by the Neimark-Sacker-Hopf bifurcation. Therefore, in the $C^1$-category, we may perform a Denjoy construction (as done in \cite{Denjoy}) for a tubular neighbourhood of the {attracting} invariant circle of any diffeomorphism to detect non-trivial wandering sets.  Our second main result is the following:

\begin{maintheorem}
\label{main_thA}
If $f$ satisfies hypotheses (P1)--(P5) and (TH), $\Sigma$ is a sufficiently small cross section to $\gamma$ and $R_0$ is the Poincar\'e map associated $\Sigma$, then there exists a diffeomorphism $G_B$ arbitrarily $C^1$-close to $R_0$, exhibiting a contracting non-trivial wandering domain $D$ and such that $\omega(D, G_B)$ is a nonhyperbolic transitive Cantor set without periodic points.
 \end{maintheorem}

The proof of Theorem \ref{main_thA} may be found in \S \ref{8.2}.  Using the Lifting Principle \cite{PR1983}, the diffeomorphisms $G_A$ and $G_B$ (given in Theorems  \ref{main_thB} and \ref{main_thA}) may be realized and then we conclude the existence of a flow $C^1$-close to that of $f$ for which strange attractors / non-trivial wandering domains may be observable. 

\begin{remark}
The authors  of \cite{CV2001, KS} found non-trivial wandering domains near a homoclinic tangency of a planar diffeomorphism by adding a series of perturbations supported in specific open sets  which are contained in disjoint gaps on the complement of persistent tangencies. Our strategy to prove Theorem \ref{main_thA} is different. 
\end{remark}

\begin{remark}
Following the ideas of \cite{CV2001, KS, LR2016}, we might think of using Theorem 2 to exhibit  historic behaviour near $\Gamma$ for a set with positive Lebesgue measure. However, from Weyl Theorem, one knows  that the orbit of each point on $\mathbb{S}^1$ by the irrational rotation is equi-distributed for the Lebesgue measure, meaning that this is not the right approach to find historic behaviour for  a set with positive Lebesgue measure. 
\end{remark}

\subsection*{Open Questions}
We finish this section with a couple of open questions. First of all, notice that the strange attractors of Theorem \ref{main_thB} and  the non-trivial wandering domains of Theorem \ref{main_thA} are found for a map which is $C^1$-close to $ R\equiv R_0$, the first return map to $S\subset \Sigma$ of a vector field $f_0$ satisfying (P1)--(P4).  At the moment, our results do not provide any information about the dynamics of $R$.
So, the first natural question is: 
\medbreak
\textbf{(Q1)} Could we obtain strange attractors and non-trivial wandering domains for $R_0$ or in a parametric family unfolding $R$?
\medbreak
In Section  \S \ref{8.1}, we use one \emph{technical hypothesis} (TH) asking for the co-existence of two heteroclinically related periodic points, just one exhibiting dominated splitting into one-dimensional sub-bundles. Another open question is: 

\medbreak
\textbf{(Q2)} Do we really need this technical hypothesis or might it be a natural consequence when one slightly moves the saddle-value $\delta>1$?
\medbreak

Although we are not able to prove, for the moment, the existence of two-dimensional strange attractors close to $\Gamma$, the scenario described by a vector field satisfying (P1)--(P5) is the natural setting in which topological two-dimensional strange attractors might occur (see \S 5.1 of \cite{FS} where periodic orbits with two positive Lyapunov exponents have been found). Therefore, we may ask:
\medbreak
\textbf{(Q3)} Could we find two-dimensional strange attractors when the cycle is broken? This configuration would give rise to what the authors of \cite{GT} call \emph{hyperchaos}.
\medbreak

We defer these tasks for future work.


\section{Return maps}
\label{Local_dynamics}
Using local coordinates near the bifocus we will provide a construction of local and global transition maps. In the end, a return map around the homoclinic cycle will be defined.
\subsection{Normal form near the bifocus}
One needs the  normal form that is used when studying the general saddle-focus case. This normal form has been constructed in Appendix A of \cite{Shilnikov et al}. 
Let $f(x, \lambda)$ as in (\ref{general})  and let $A$ and $B$ the  $(2\times 2)$-matrices as in \cite{Shilnikov et al} that depend on the parameter $\lambda$. It is clear that
$$
 A(0)= \begin{pmatrix} -\alpha_1 & \omega_1 \\ \omega_1 &-\alpha_1  \end{pmatrix} \qquad \text{and} \qquad B(0)= \begin{pmatrix} \alpha_2 & \omega_2 \\ \omega_2 &\alpha_2  \end{pmatrix}.
$$
\bigbreak
The generalization of Bruno's theorem may be stated as:
\begin{theorem}[Shilnikov \emph{et al} \cite{Shilnikov et al}, adapted]
\label{normal form1}
There is a local $C^{5}$ transformation near $O$ such that in  the new
coordinates $(x,y)=(x_1,x_2,x_3, x_4),$ the system casts as follows
\begin{equation}
\label{eq1}
\left\{ 
\begin{array}{l}
\dot x = A(\lambda) x + f(x,y, \lambda)x, \\ 
\dot y = B(\lambda) y + g(x,y, \lambda)y
\end{array}
\right.
\end{equation}
where: 
\begin{itemize}
\item  $A(\lambda)$ and $B(\lambda)$ are $(2\times 2)$-matrices functions; 
\medbreak
\item $f,g$ are $C^{5}$-smooth with respect to $(x,y)$, their first derivatives are $C^{4}$-smooth with respect to $(x,y, \lambda)$ and 
\medbreak
\item the following identities are valid for every $x=(x_1, x_2)$, $y=(y_1, y_2)$ and $\lambda \approx 0$:
$$
f(0,0, \lambda)=0,\; \quad g(0,0, \lambda)=0,\; \quad  f(x,0,\lambda)=0,\; \quad  g(0,y, \lambda)=0
$$
and 
$$
f(0,y, \lambda)=0,\; \quad g(x,0, \lambda)=0.
$$
\end{itemize}
\end{theorem}

 Without loss of generality, we are assuming that the neighbourhood $V_O$ in which the flow can be $C^5$-linearized near $O$ is a solid hypertorus. Rescaling the local coordinates, the solid hypertorus can be considered as the product of two unitary disks.

\medbreak

\subsection{Local coordinates near $O$}
Let us consider bipolar coordinates $(r_s, \phi_s,r_u, \phi_u) \in [0,1] \times \RR \pmod{2\pi}\times [0,1] \times \RR\pmod{2\pi}$ on $V_{O}$ such that
$$
x_1=r_s \cos(\phi_s), \qquad x_2=r_s \sin(\phi_s), \qquad x_3=r_u \cos(\phi_u)\qquad \text{and} \qquad
x_4=r_u \sin(\phi_u).
$$
In these coordinates the local invariant manifolds are given by
$$
W^s_{loc}(O) \equiv \{ r_u=0\} \qquad \text{and} \qquad   W^u_{loc}(O) \equiv \{ r_s=0\}
$$
and, up to high order terms,  we can rearrange system (\ref{eq1}) as
\begin{equation}
\label{equation1}
\dot{r}_s=-\alpha_1 r_s, \qquad \dot{\phi}_s=\omega_1, \qquad \dot{r}_u=\alpha_2 r_u \qquad \text{and} \qquad  \dot{\phi}_u=\omega_2.
\end{equation}
Solving the above system explicitly we get
$$
r_s(t)=r_s(0) e^{-\alpha_1 t}\qquad\phi_s(t)=\phi_s(0)+ \omega_1 t   \qquad  r_u(t)=r_u(0) e^{\alpha_2 t} \qquad   \phi_u(t)=\phi_u(0)+ \omega_2t.
$$

\subsection{Cross sections near $O$}
In order to construct a first return map around the homoclinic cycle $\Gamma$ we consider two solid tori $\Sigma^{in}_O$ and $\Sigma^{out}_O$ defined by
\begin{enumerate}
\item[(a)] $\Sigma^{in}_O\equiv~\{r_s=1 \}$ with coordinates  $(\phi_s^{in}, r_u^{in}, \phi_u^{in})$,
\medbreak
\item[(b)] $\Sigma^{out}_O\equiv \{r_u=1 \}$  with coordinates $(r_s^{out}, \phi_s^{out}, \phi_u^{out}).$
\end{enumerate}
These sets, depicted in Figure \ref{cross_sections}, are transverse to the flow.
\bigbreak
By construction, trajectories starting at interior points of $\Sigma^{in}_O$ go inside the hypertorus $V_O$ in positive time. Trajectories starting at interior points of $\Sigma^{out}_O$ go outside $V_O$ in positive time.
Intersections between local invariant manifolds at $O$ and cross sections are circles parametrized as
\begin{equation}
\label{local s O}
W^s_{loc}(O) \cap \Sigma^{in}_O=\{r_u^{in}=0 \quad  \text{and } \quad 0\leq   \phi^{in}_s< 2\pi\}
\end{equation}
and
\begin{equation}
\label{local u O}
W^u_{loc}(O) \cap \Sigma^{out}_O=\{r_s^{out}=0 \quad  \text{and } \quad 0\leq \phi^{out}_u < 2\pi \}.
\end{equation}

\begin{figure}
\begin{center}
\includegraphics[height=6.9cm]{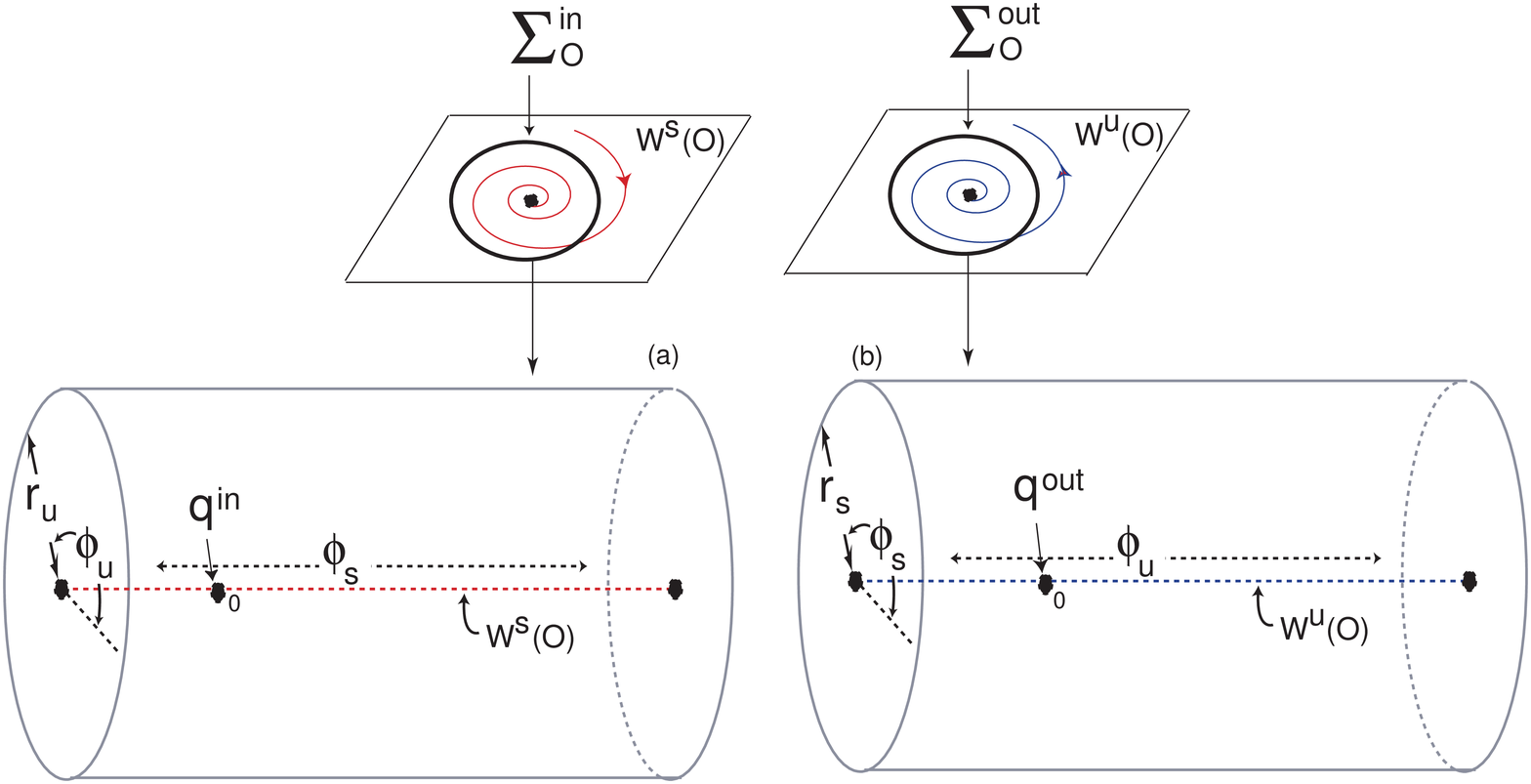}
\end{center}
\caption{\small Coordinates and cross sections near $O$: (a) $\Sigma^{in}_O$ and (b) $\Sigma^{out}_O$. Superscripts \emph{in} and \emph{out} are omitted.}
\label{cross_sections}
\end{figure}

\subsection{Local transition maps near $O$}
\label{local near O}
The time of flight from $\Sigma^{in}_O$ to $\Sigma^{out}_O$ of a trajectory with initial condition $(\phi_s^{in}, r_u^{in}, \phi_u^{in}) \in \Sigma^{in}_O\backslash W^s_{loc}(O)$ only depends on the coordinate $r_u^{in}> 0$ and is given by
$$
T(\phi_s^{in}, r_u^{in}, \phi_u^{in}) =-\frac{\ln(r_u^{in})}{\alpha_2}>0.
$$
Since $r_u^{in} >0$, $T$ is well defined and non-negative. Hence the local map
$$
\Pi_{O}: \Sigma^{in}_O \backslash W^s_{loc}(O) \rightarrow \Sigma^{out}_O
$$
is given by
\begin{equation}
\label{local_flow_eq}
\left(\begin{array}{l}\phi_s^{in}\\ \\r_u^{in}\\ \\\phi_u^{in}\end{array}\right)
\mapsto
\left(\begin{array}{l}r_s^{out}\\ \\\phi_s^{out} \\ \\ \phi_u^{out}\end{array}\right)
=
\left(\begin{array}{l}
\left(r_u^{in}\right)^{\frac{\alpha_1}{\alpha_2}}
\\
\\
\phi_s^{in}-\frac{\omega_1}{\alpha_2} \ln (r_u^{in}) \pmod{2\pi}
\\
\\
\phi_u^{in}-\frac{\omega_2}{\alpha_2} \ln (r_u^{in}) \pmod{2\pi}
\end{array}\right).
\end{equation}
Since $\delta:=\frac{\alpha_1}{\alpha_2}>1$ (see (P2)), the flow is volume-contracting near $O$. 

\subsection{The inverse of $\Pi_O$}
It follows from (\ref{local_flow_eq}) that the inverse of the local transition map
$$
\Pi_O^{-1}:\Sigma_O^{out}\setminus W^{u}_{loc}(O) \rightarrow \Sigma_O^{in}
$$
can be written as
\begin{equation}
\label{inverse_local_flow_eq}
\left(\begin{array}{l}r_s^{out}\\ \\\phi_s^{out} \\ \\ \phi_u^{out}\end{array}\right)
\mapsto
\left(\begin{array}{l}\phi_s^{in}\\ \\r_u^{in}\\ \\\phi_u^{in}\end{array}\right)
=
\left(\begin{array}{l}
\phi_s^{out}+\frac{\omega_1}{\alpha_1} \ln (r_s^{out}) \pmod{2\pi}
\\
\\
\left(r_s^{out}\right)^{\frac{\alpha_2}{\alpha_1}}
\\
\\
\phi_u^{out}+\frac{\omega_2}{\alpha_1} \ln (r_s^{out}) \pmod{2\pi}
\end{array}\right).
\end{equation}

\subsection{Global transition the return map}
\label{global transition}
Here, we define the global map from $\Sigma_O^{out}$ to $\Sigma_O^{in}$ corresponding to a flow-box around the homoclinic connection $\gamma$. \bigbreak
 By taking $V_{O}$ small enough, we can assume that $\gamma$ intersects each one of the cross sections $\Sigma_O^{in}$ and $\Sigma_O^{out}$ at exactly one point, $q^{in}$ and $q^{out}$,  defined by:
$$
\{q^{in}\}=\gamma \cap \Sigma^{in}_O \qquad \mbox{and} \qquad \{q^{out}\}=\gamma \cap \Sigma^{out}_O,
$$
as shown in Figure \ref{cross_sections}. Without loss of generality, we may assume that the $\phi_s^{in}$, $\phi_u^{out}$  coordinates of $q^{in}$ and $q^{out}$ are zero.
\bigbreak
Therefore, there exists $T_1>0$ such that $\varphi(T_1,q^{out})=q^{in}$ and $\varphi([0,T_1[,q^{out})\cap \Sigma_{O}^{in}=\varnothing$. Using the regularity of the flow and the Tubular Flow Theorem \cite{PM}, it follows that, given any neighbourhood $C^{in} \subset \Sigma_{O}^{in}$ of $q^{in}$, there exist a neighbourhood $C^{out} \subset \Sigma_{O}^{out}$ of $q^{out}$ and $\tau_1: C^{out}\rightarrow \mathbb{R}$ such that:
\medbreak
\begin{itemize}
\item $\tau_1$ is a $C^r$ map, $r\geq 5$;
\medbreak
\item $\tau_1(q^{out})=T_1$;
\medbreak
\item $\varphi(\tau_1(q),q)\in C^{in}$ for all $q\in C^{out}$.
\end{itemize}
\bigbreak
Now, we define the global map $\Psi_1:C^{out} \rightarrow C^{in}$ as
$\Psi_1(q)=\varphi(\tau_1(q),q)$, with $q \in C^{out}$.
 The map $\Psi_1$ represents the global reinjection from $\Sigma_O^{out}$ to $\Sigma_O^{in}$ following $\Gamma$. 
Then we consider the set $S=\Pi_{O}^{-1}(C^{out}\setminus W_{loc}^{u}(O)) \subset \Sigma_{O}^{in}$ and define the first return map
$
R_{0}: S\rightarrow C^{in}  
$
as
\begin{equation}
\label{F1}
R_{0}|_{S}=\Psi_1 \circ \Pi_{O}.
\end{equation}
The location of $S$ is sketched in Figure \ref{torus}.
Note that $R_{0}$ is of class $C^r$, with $r\geq 5$, and is well defined. See also \cite{FS, IbRo}.
\medbreak

\begin{figure}[ht]
\begin{center}
\includegraphics[height=4.9cm]{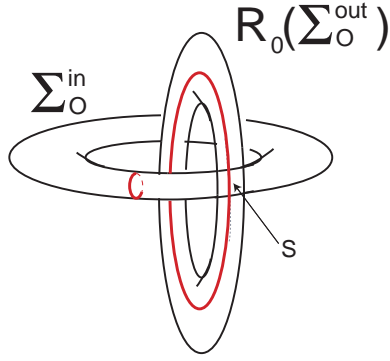}
\end{center}
\caption{\small Global geometry around the bifocus. }
\label{torus}
\end{figure}

\subsection{Technical notation}
Without loss of generality, we assume that the neighbourhoods  $C^{in}$ and $C^{out}$ introduced above, may be parameterised as:
\begin{eqnarray*}
C^{in}=\left\{ (\phi_s^{in}, r_u^{in}, \phi_u^{in}):
\phi_s^{in} \in [-\varepsilon^{in},\varepsilon^{in}],
r_u^{in}\in [0,c^{in}],
\phi_u^{in} \in [0,2\pi] \right\} \subset \Sigma_O^{in},
\end{eqnarray*}
for some small constants $0<\varepsilon^{in}, c^{in}\leq 1$, and
\begin{equation}
\label{ciout}
C^{out}=\left\{ (r_s^{out}, \phi_s^{out}, \phi_u^{out}):
r_s^{out}\in [0,c^{out}],
\phi_s^{out} \in [0,2\pi], \phi_u^{out} \in [-\varepsilon^{out},\varepsilon^{out}] \right\} \subset \Sigma_O^{out}
\end{equation}
for some small enough constants  $0<\varepsilon^{out}, c^{out}\leq 1$.

\subsection{A model for the transition}
Up to high order terms, the global map $\Psi_1:C^{out} \rightarrow C^{in}$ may be given, in rectangular coordinates, by:
\begin{eqnarray*}
\label{global_flow_eq}
\left(\begin{array}{c}\overline{X}\\ \\ \overline{Y}\\ \\ \overline{Z}\end{array}\right)
&\mapsto&
A
\left(\begin{array}{c}r_s^{out} \cos(\phi_s^{out})\\ \\  r_s^{out} \sin(\phi_s^{out}) \\ \\ \phi_u^{out}\end{array}\right) + \ldots
\end{eqnarray*}
where $A$ is a linear map such that $\det A \neq 0$ and where the dots represent the high order terms effects of $\lambda$ (see (2.6) of \cite{FS}) . A simple choice of $A$ compatible with hypothesis (P4) is
$
 \left(\begin{array}{lll}0 & 1 & 0\\ 0 & 0 & 1 \\ 1 & 0 & 0\end{array}\right)
$
and thus the return map $R_0$ to $C^{in}\subset \Sigma_O^{in}$ may be written as:
\begin{eqnarray}
\label{global_flow_eq1}
\left(\begin{array}{l}\overline{X}\\ \\ \overline{Y}\\ \\ \overline{Z} \end{array}\right)
&\mapsto&
 \left(\begin{array}{lll}0 & 1 & 0\\ 0 & 0 & 1 \\ 1 & 0 & 0\end{array}\right)
\left(\begin{array}{c}r_s^{out} \cos( \phi_u^{out}) \\ \\ r_s^{out} \sin( \phi_u^{out}) \\ \\ \phi_u^{out}\end{array}\right)+ \ldots.
\end{eqnarray}
 Taking into account \eqref{local_flow_eq}, the previous map is equivalent to:
\begin{eqnarray*}
\label{global_flow_eq2}
\left(\begin{array}{l} \overline{X}\\ \\ \overline{Y}\\ \\ \overline{Z}   \end{array}\right)
&\mapsto&
 \left(\begin{array}{lll}0 & 1 & 0\\ 0 & 0 & 1 \\ 1 & 0 & 0\end{array}\right)
\left(\begin{array}{c} \left(r_u^{in}\right)^\delta \cos( \phi_s^{in}-\frac{\omega_2}{\alpha_2}\ln (r_u^{in})) \\ \\ \left(r_u^{in}\right)^\delta \sin( \phi_s^{in}-\frac{\omega_1}{\alpha_2}\ln (r_u^{in}))\\ \\  \phi_u^{in}-\frac{\omega_1}{\alpha_2}\ln (r_u^{in})\end{array}\right). 
\end{eqnarray*}
Since $(r_s, \phi_s,r_u, \phi_u)$ are bipolar coordinates in $\Sigma_O^{in}= \{ r_s=1\}$, we get:
$$
(r_u^{in})^2 = X^2+Y^2, \qquad  \phi_u^{in} = \arctan \left( \frac{Y}{X}\right) + k\pi,  k\in \ZZ \qquad \text{and} \qquad  \phi_s^{in} = {Z}+ \ldots ,
$$
and thus we may write an explicit expression for $R_0$:
\begin{equation}
\label{global_transition1}
\left(\begin{array}{l}X\\ \\ Y\\ \\ Z\end{array}\right)
\mapsto
 \left(\begin{array}{c} \left(X^2+Y^2\right)^\frac{\delta}{2} \sin\left( Z -\frac{\omega_2}{ 2\alpha_2}\ln \left(X^2+Y^2\right)\right) \\ \\  \arctan \left( \frac{Y}{X}\right) -\frac{\omega_1}{2 \alpha_2}\ln (X^2+Y^2) \\ \\   \left(X^2+Y^2\right)^\frac{\delta}{2} \cos\left( Z -\frac{\omega_2}{ 2\alpha_2}\ln \left(X^2+Y^2\right)\right)\end{array}\right) + \ldots,\end{equation}
a model similar to equations (2.6) and (3.6) of \cite{FS}.

\bigbreak
\begin{remark}
Our construction holds if the global map considered in \eqref{global_flow_eq1} is another model compatible with hypotesis (P4). 
\end{remark}

\section{Local geometry near the cycle }
\label{Local_Geom}

\subsection{Notions related with spiralling behaviour}
In this section, we introduce the notions of segment, spiral, helix, spiralling sheet and scroll.  These definitions are adapted from \cite{Hart, IbRo}.
\bigbreak
\begin{definition}
A segment $s$ in $\Sigma_O^{in}$ is a regular curve  $s:[0,1] \rightarrow \Sigma_O^{in}$ parametrized by $t$ that meets $W^s_{loc}(O)$ transversely and only at a point $s(0)$ and such that writing $s(t)= (\phi_s^{in}(t), r_u^{in}(t), \phi_u^{in}(t))$, then: 
\begin{itemize}
\item the components are monotonic functions of $t$ and
\item $\phi_s^{in}(t)$ and $ \phi_u^{in}(t)$ are bounded.
\end{itemize}
Similarly, we define a segment in $\Sigma_O^{out}$.
\end{definition}

\bigbreak

\begin{definition}
Let $a \in \RR$, $D$ be a disc centered at $p\in\RR^2$.
A \emph{spiral} on $D$ around the point $p$ is a smooth curve
$\alpha :[a, +\infty[ \rightarrow D,$ satisfying $ \lim_{s\to +\infty }\alpha (s)=p$ and such that if
$\alpha (s)=(r(s),\theta(s))$ is its expression in polar coordinates around $p$ then:
\begin{enumerate}
\item the map $r$ is bounded by two monotonically decreasing maps converging to zero as $s \rightarrow +\infty$;
\item the map $\theta$ is monotonic for some unbounded subinterval of $[a,+\infty[$ and
\item $\lim_{s\to +\infty}|\theta(s)|=+\infty$.
\end{enumerate}
\end{definition}

\bigbreak
The notion of spiral may be naturally extended to any set diffeomorphic to a disk. 
\bigbreak
\begin{definition}
A helix $H\subset\Sigma_O^{in}$ accumulating on $W_{loc}^s(O)$ is a curve (parametrized by $t$)
$$
H: [0,1] \rightarrow \Sigma_O^{in}
$$
without self-intersections such that if $H(t)= (\phi_s^{in}(t), r_u^{in}(t), \phi_u^{in}(t))$, then: 
\begin{itemize}
\item the components are quasi-monotonic functions of $t$;
\item $\lim_{t \rightarrow 0^+} |\phi_s^{in}(t)| = \lim_{t \rightarrow 0^+} |\phi_u^{in}(t)| = +\infty$ and $\lim_{t \rightarrow 0^+} r_u^{in}(t)=0$.
\end{itemize}
Similarly, we define a  helix in $\Sigma_O^{out}$ accumulating on $W^u_{loc}(O)$. 
\end{definition}

\bigbreak

\begin{definition}
\label{ss_def}
A two-dimensional manifold  $\mathcal{H}$ embedded in $\mathbb{R}^3$ is called a \emph{spiralling sheet} accumulating on a curve $\mathcal{C}$ if there exist a spiral $S$ around $(0,0)$, a neighbourhood $V\subset \mathbb{R}^3$ of $\mathcal{C}$, a neighbourhood $W_0 \subset \mathbb{R}^2$ of the origin, a non-trivial closed interval $I$ and a diffeomorphism $\eta: V\rightarrow I \times W_0$ such that:
$$\eta(\mathcal{H}\cap V)=I \times (S \cap W_0) \qquad \text{and} \qquad \gamma=\eta^{-1}(I\times \{0\}).$$
\end{definition}
\bigbreak
The curve $\mathcal{C}$ can be called the \emph{basis} of the spiralling sheet. According to Definition \ref{ss_def}, up to a diffeomorphism, we may think on a \emph{spiralling sheet} accumulating on a curve as the cartesian product of a spiral and a curve.  In the present paper, the curve $\mathcal{C}$ lies on the invariant manifolds of $O$.
Each transverse cross section to $\mathcal{C}$ intersects the spiralling sheet into a spiral. Note also that the diffeomorphic image of a spiralling sheet is again a spiralling sheet.
\bigbreak
\begin{definition}
\label{scroll_def}
Given two spiralling sheets $\mathcal{H}_1$ and $\mathcal{H}_2$ accumulating on the same curve $\mathcal{C}\subset\mathbb{R}^3$, any region limited by $\mathcal{H}_1$ and $\mathcal{H}_2$ inside a tubular neighbourhood of $\mathcal{C}$ is said a \emph{scroll} accumulating on $\mathcal{C}$.
\end{definition}
\bigbreak

\subsection{Local geometry}
\label{firstreturn}
The following result will be essential  in the sequel. It gives a general characterization of the geometry near the bifocus. See Figure \ref{local geometry}.
\begin{proposition}[\cite{Hart, IbRo}, adapted]
\label{My_Prop5}
For $\xi>0$ arbitrarily small, let $\Xi: D\subset \mathbb{R}^2 \rightarrow \RR$ be a $C^1$ map defined on the disk
$$D=\{(u,v)\in\mathbb{R}^2 : 0\leq u^2+v^2 \leq \xi <1\}$$ and let
$$
\mathcal{F}^{in}=\{(\phi_s^{in}, r_u^{in}, \phi_u^{in}) \in \Sigma^{in}_O : \phi_s^{in}=\Xi(r_u^{in} \cos \phi_u^{in},r_u^{in} \sin \phi_u^{in}),\,0\leq r_u^{in}\leq \xi,\,0\leq \phi_u^{in} <2\pi\}
$$
and
$$
\mathcal{F}^{out}=\{(r_s^{out},\phi_s^{out}, \phi_u^{out}) \in \Sigma^{out}_O : \phi_u^{out}=\Xi(r_s^{out} \cos \phi_u^{out},r_s^{out} \sin \phi_u^{out}),\,0\leq r_s^{out}\leq \xi,\,0\leq \phi_s^{out} <2\pi\}.
$$
Then the following assertions are valid:
\medbreak
\begin{enumerate}
\item Any segment in $\mathcal{F}^{in}\backslash W^s_{loc}(O)$ is mapped by $\Pi_O$ into a helix accumulating on $W^u_{loc}(O)$.
\medbreak
\item The set $\Pi_{O}(\mathcal{F}^{in}\backslash W^s_{loc}(O))$ is a spiralling sheet accumulating on $W^u_{loc}(O) \cap \Sigma^{out}_O$. 

\medbreak
\item The set $\Pi_{O}^{-1}(\mathcal{F}^{out}\backslash W^u_{loc}(O))$ is a spiralling sheet accumulating on $W^s_{loc}(O) \cap \Sigma^{in}_O$. \medbreak
\item The set $S$  introduced in \S\ref{global transition} is a scroll contained in $\Sigma_O^{in}$ accumulating on $W^{s}_{loc} (O)\cap \Sigma_O^{in}$. 
\medbreak
\end{enumerate}
 \end{proposition}

\begin{figure}[ht]
\begin{center}
\includegraphics[height=5.0cm]{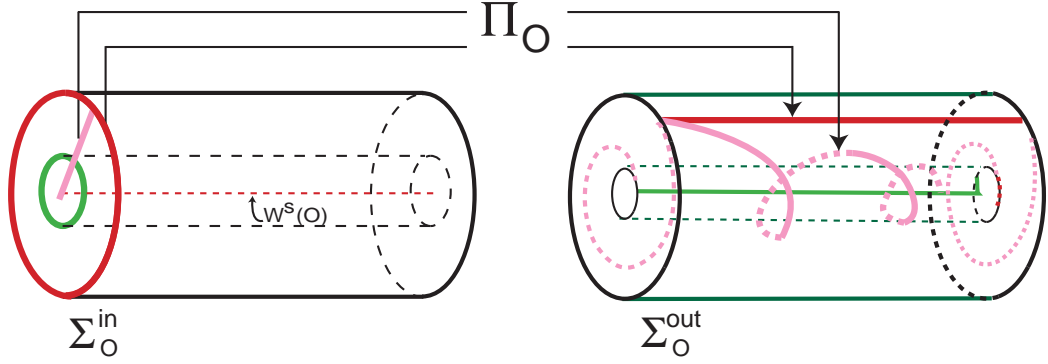}
\end{center}
\caption{\small Any segment in $\mathcal{F}^{in}\backslash W^s_{loc}(O)$ is mapped by $\Pi_O$ into a helix accumulating on $W^u_{loc}(O)$. Any closed curve in $\Sigma_O^{in}$ is mapped by $\Pi_O$ into a closed curve in $\Sigma_O^{out}$. The red curve in this Figure is supposed to be the same as the red curve of Figure \ref{torus}.}
\label{local geometry}
\end{figure}

\section{Three-dimensional whiskered horseshoes revisited}
\label{horseshoes}
The existence of a homoclinic cycle $\Gamma$ is considered as a mechanism to create three-dimensional chaos in the spirit of Shilnikov \cite{Shilnikov65, Shilnikov67A} and Lerman \cite{Lerman2000}.  In this section, we recall the main steps of the construction of the invariant horseshoe given in \cite{IbRo}, adapted to our purposes. We address the reader to \cite{Wiggins} for more details in the definitions.
\subsection{Geometric preliminaries for the construction}
\label{slices and slabs}
Let $D\subset \RR^3$ be a compact and connected 3-dimensional set of $\RR^3$. 
Define $D_X$ and $D_y$ the projection of $D$ onto $\RR^2$ and $\RR$ as:
$$D_X=\{X\in \RR^2: \text{ for which there exists } y\in \RR \text{ with } (X, y)\in D\}\subset \RR^2$$
and
$$D_y=\{y\in \RR: \text{ for which there exists } x\in \RR^2 \text{ with } (X, y)\in D\}\subset \RR. $$
In our case, as illustrated in Figure \ref{box},  $D_X$ is a closed and connected two-dimensional square contained in $\RR^2$ and $D_y$ is a bounded interval of $\RR$. 

\bigbreak

\begin{definition}
A $\mu_h$-horizontal slice $\mathcal{H}$ is defined to be the graph of a function $h: D_X \rightarrow \RR$ satisfying:
\medbreak
\begin{itemize}
\item $\mathcal{H}= \{(X, h(X))\in \RR^3: X \in D_X\} \subset D$ and
\medbreak
\item for all $ X_1, X_2 \in D_X$, there exists $\mu_h\in \RR^+_0$ such that  $|h(X_1)-h(X_2)|\leq \mu_h \|X_1-X_2\|$.
\end{itemize}
\bigbreak
A $\mu_v$-vertical slice $\mathcal{V}$ is defined to be the graph of a function $v: D_y \rightarrow \RR^2$ satisfying:
\begin{itemize}
\medbreak
\item $\mathcal{H}= \{(v(y), y)\in \RR^3: y \in I_y\} \subset D$ and
\medbreak
\item for all $ y_1, y_2 \in D_y$, there exists $\mu_v\in \RR^+_0$ such that  $\|v(y_1)-v(y_2)\|\leq \mu_v |y_1-y_2|$.
\end{itemize}
\end{definition}
\bigbreak

Let $\mu_h>0$ fixed, let $H\subset D$ be a $\mu_h$-horizontal slice and let $D_y\subset D $ be an interval intersecting $H$ at one point. 
Now, consider the following set:
$$
S_H = \{ (X,y) \in \RR^2\times \RR: X \in D_X \text{ and } y \text{ has the Property (P) }\}
$$
where:
\medbreak
 \textbf{Property (P): } for each $X\in D_X$, given any line $L$ through $(X,y)$ with L parallel to the plane $X=0$, then $L$ intersects the points $(X, h_\alpha (X))$ and $(X, h_\beta (X))$ for some $\alpha, \beta \in D_y$ with $(X, y)$ between these two points along $L$. 
\medbreak
\begin{definition}
A $\mu_h$\emph{-horizontal slab} is defined to be the topological closure of $S_H$. See Figure \ref{box}.
 \end{definition}
 
  \begin{figure}
\begin{center}
\includegraphics[height=6.9cm]{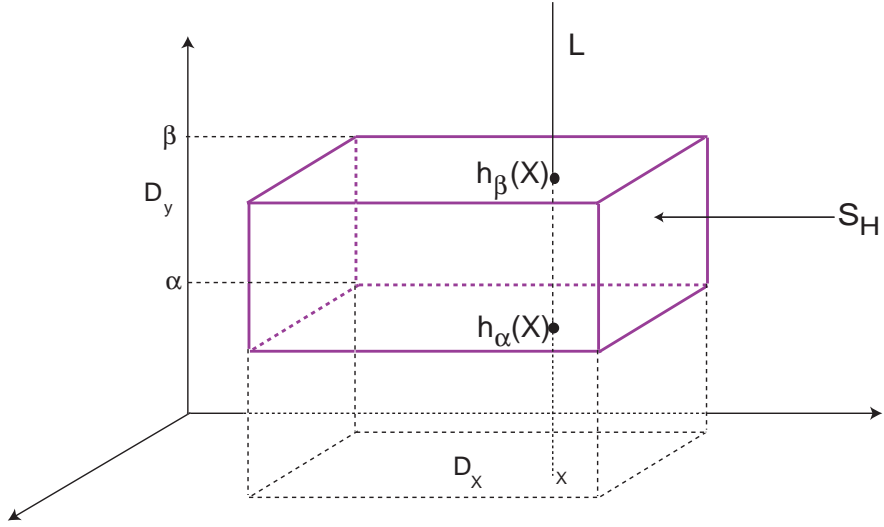}
\end{center}
\caption{\small Illustration of the sets $D_X$, $D_y$, $\overline{S_H}$ (horizontal slab) and Property (P).}
\label{box}
\end{figure}

\bigbreak
The \emph{vertical boundary} of a $\mu_h$-horizontal slab $S_H$, denoted by $\partial_v S_H$, is defined as 
$$\partial_v S_H = \{(X,y)\in S_H: X\in \partial D_X\}.$$

The horizontal boundary of a $\mu_h$-horizontal slab $S_H$, denoted by $\partial_h S_H$, is defined by $$\partial_h H = \partial S_H - \partial_v S_H.$$ Vertical slabs and their boundaries may be defined similarly. 

\begin{definition}
Let $S_H^1$ and $S_H^2$ be two $\mu_h$-horizontal slabs. We say that $S_H^1$ intersects $S_H^2$ \emph{fully} if $S_H^1\subset S_H^2$ and  $\partial_v S_H^1\subset \partial_v S_H^2$.
\end{definition}

\begin{definition}
The width of a $\mu_h$-horizontal slab $S_H$, denoted by $d(S_H)$, is defined as:
\begin{equation}
\label{dist}
d(S_H)= \sup_{X\in D_X, \alpha, \beta \in I} |h_\alpha (X)-h_\beta(X)|
\end{equation}
Similarly, we define the width of a $\mu_v$-vertical slab. More details in \S 2.3 of \cite{Wiggins}.
\end{definition}

\subsection{The construction}
\label{suspended horseshoes}
In this section we focus our attention on the dynamics of 
\begin{equation}
\label{def_S}
S=\Pi_{O}^{-1}(C^{out}\setminus W_{loc}^{u}(O)) \subset C^{in}\subset \Sigma_{O}^{in}
\end{equation}
  defined in Subsection \ref{global transition}. To simplify the readers' task we revisit few results, including their proofs, whose arguments will be required in the sequel.
We also discuss the existence of invariant sets in $S$, accumulating on $W_{loc}^s(O) \cap \Sigma_O^{in}$, where $R_0$ is topologically conjugate to a shift under (at least) two symbols. The construction is based on the generalized Conley-Moser conditions \cite{Koon, Wiggins}, which provide sufficient conditions for the existence of invariant sets where the dynamics is conjugated to a full shift. 
\bigbreak

Let $C^{out}\subset \Sigma_O^{out}$ be the solid cylinder of radius $c^{out}$ given by (\ref{ciout}). Given $\eta\in [0,2\pi]$ and $\varepsilon^{out}>0$, for each $N \in \mathbb{N}$  such that $N \geq -\frac{1}{2\pi}\left(\frac{\omega_2}{\alpha_2}\ln(c^{out})+\eta\right),$ we define the hollow cylinder
\begin{eqnarray*}
M_{N}^{out}=\left\{ (r_s^{out}, \phi_s^{out}, \phi_u^{out}):
r_s^{out}\in [a_{N+1},a_N],
\phi_s^{out} \in [0,2\pi], \phi_u^{out} \in [-\varepsilon^{out},\varepsilon^{out}] \right\} \subset C^{out},
\end{eqnarray*}
where
\begin{equation}
\label{a_n}
a_{N}=\exp\left(\frac{-\alpha_1(\eta+2\pi N)}{\omega_2}\right).
\end{equation}
As depicted in Figure \ref{annulus}, the border of $M_{N}^{out}$, denoted by $\partial M_{N}^{out}$, can be written as
$$
\partial M_{N}^{out} = E_N^{L} \cup E_N^{R}  \cup T_N^{I} \cup T_N^{O},
$$
(the letters $L$, $R$, $I$, $O$ mean Left, Right, Inner and Outer, respectively) with: 
\begin{eqnarray*}
E_N^{L}&=&\left\{ (r_s^{out}, \phi_s^{out}, \phi_u^{out}):
r_s^{out}\in [a_{N+1},a_N],
\phi_s^{out} \in [0,2\pi], \phi_u^{out} =-\varepsilon^{out} \right\},
\\
E_N^{R}&=&\left\{ (r_s^{out}, \phi_s^{out}, \phi_u^{out}):
r_s^{out}\in [a_{N+1},a_N],
\phi_s^{out} \in [0,2\pi], \phi_u^{out}= \varepsilon^{out} \right\},
\\
T_N^{I}&=&\left\{ (r_s^{out}, \phi_s^{out}, \phi_u^{out}):
r_s^{out}=a_{N+1},
\phi_s^{out} \in [0,2\pi], \phi_u^{out} \in [-\varepsilon^{out},\varepsilon^{out}] \right\} \quad \text{and }
\\
T_N^{O}&=&\left\{ (r_s^{out}, \phi_s^{out}, \phi_u^{out}):
r_s^{out}=a_N,
\phi_s^{out} \in [0,2\pi], \phi_u^{out} \in [-\varepsilon^{out},\varepsilon^{out}] \right\}.
\end{eqnarray*}

  \begin{figure}[ht]
\begin{center}
\includegraphics[height=4.9cm]{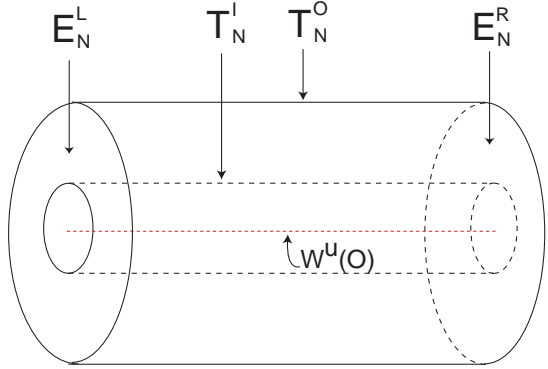}
\end{center}
\caption{\small  The border of $M_{N}^{out}$, denoted by $\partial M_{N}^{out}$, can be written as
$\partial M_{N}^{out} = E_N^{L} \cup E_N^{R}  \cup T_N^{I} \cup T_N^{O}$.}
\label{annulus}
\end{figure}

\bigbreak

According to the definitions given in Subsection \ref{slices and slabs}, it is easy to check the set $M_{N}^{out}$ is a horizontal slab across $\Sigma_O^{out}$. The vertical (resp. horizontal) boundaries of $M_{N}^{out}$ are defined by $E_N^L$ and $E_N^R$ (resp. $T_N^I$ and $T_N^O$). The surfaces $T_N^I$ and $T_N^O$ may be defined as graphs of functions $r_u^{in}=h(\phi_s^{in},\phi_u^{in})$ where $h$ is approximately a constant map.  Define now:
\begin{equation}
\label{sn}
 \mathcal{S}_{N}=\Pi_O^{-1}(M_{N}^{out})\subset \Sigma_O^{in}.
 \end{equation}
  Using the coordinates of the different components of $\partial M_{N}^{out}$, the authors of \cite{IbRo} proved that: 

\begin{eqnarray*}
\Pi^{-1}_O(E_N^{L})= & \{ (\phi_s^{in}, r_u^{in},  \phi_u^{in}):
&       r_u^{in}=\exp\left( \frac{\alpha_2 (\phi_u^{in}+\varepsilon^{out})}{\omega_2}\right),
\\
& &     \phi_u^{in} \in [-\varepsilon^{out}-\eta-2\pi (N+1),-\varepsilon^{out}-\eta-2\pi N],
\\
& &     \phi_s^{in} \in [0, 2\pi] \} \subset \Sigma_O^{in}\},
\\
\\
\Pi^{-1}_O(E_N^{R})= &  \{ (\phi_s^{in}, r_u^{in},  \phi_u^{in}):
&       r_u^{in}=\exp\left( \frac{\alpha_2 (\phi_u^{in}-\varepsilon^{out})}{\omega_2}\right),
\\
& &     \phi_u^{in} \in [\varepsilon^{out}-\eta-2\pi (N+1),\varepsilon^{out}-\eta-2\pi N],
\\
& &     \phi_s^{in} \in [0, 2\pi] \} \subset \Sigma_O^{in}\},
\\
\\
\Pi^{-1}_O(T_N^{I})= &  \{ (\phi_s^{in}, r_u^{in},  \phi_u^{in}):
&       r_u^{in}=\exp\left( -\frac{\alpha_2 (\gamma+2\pi(N+1))}{\omega_2}\right),
\\
& &     \phi_u^{in} \in [-\varepsilon^{out}-\eta-2\pi (N+1),\varepsilon^{out}-\eta-2\pi (N+1)],
\\
& &     \phi_s^{in} \in [0, 2\pi] \} \subset \Sigma_O^{in}\} \qquad \text{and}
\\
\\
\Pi^{-1}_O(T_N^{O})= &  \{ (\phi_s^{in}, r_u^{in},  \phi_u^{in}):
&       r_u^{in}=\exp\left( -\frac{\alpha_2 (\gamma+2\pi N)}{\omega_2}\right),
\\
& &     \phi_u^{in} \in [-\varepsilon^{out}-\eta-2\pi N,\varepsilon^{out}-\eta-2\pi N],
\\
& &     \phi_s^{in} \in [0, 2\pi] \} \subset \Sigma_O^{in}\}.
\end{eqnarray*}

\bigbreak
By Proposition \ref{My_Prop5}, the set $S=\Pi_O^{-1}(C^{out})\subset \Sigma_O^{in}$ is a scroll accumulating on $W^s_{loc}(O)\cap \Sigma_O^{in}$ because each of the two disks 
\begin{eqnarray*}
\left\{ (r_s^{out}, \phi_s^{out}, \phi_u^{out}):
r_s^{out}\in [0, c^{out}],\,\,
\phi_s^{out} \in [0,2\pi],\,\, \phi_u^{out} =\pm \varepsilon^{out} \right\} 
\end{eqnarray*}
limiting $C^{out}$ is sent by $\Pi_O^{-1}$ into a spiralling sheet accumulating on $W^s_{loc}(O)\cap \Sigma_O^{in}$. Therefore $S$ is limited by two spiralling sheets accumulating on $W^s_{loc}(O)\cap \Sigma_O^{in}$.

  \begin{figure}[ht]
\begin{center}
\includegraphics[height=4.1cm]{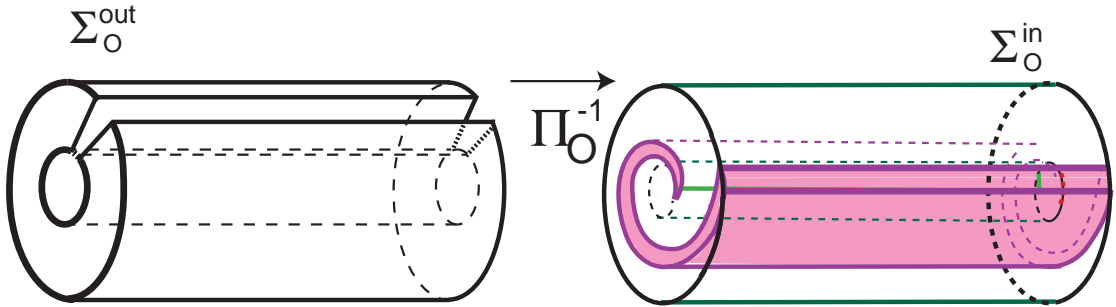}
\end{center}
\caption{\small Each of the two disks limiting $M_N^{out~}\subset C^{out}$ is sent by $\Pi_O^{-1}$ into a spiralling sheet accumulating on $W^s_{loc}(O)\cap \Sigma_O^{in}$.}
\label{local_map2}
\end{figure}

\bigbreak
The family of sets $\mathcal{S}_{N}$ (see \eqref{sn}) provides an infinite collection of pieces inside the scroll $S$ accumulating on $W^s_{loc}(O)\cap \Sigma_O^{in}$.
Note that $\mathcal{S}_{N}$ is limited by two tori contained in $\Sigma_O^{in}$. More precisely
$$
\mathcal{S}_{N} \subset \left\{ (\phi_s^{in},r_u^{in},\phi_u^{in}) \in \Sigma_O^{in}:
b_{N+1}\leq r_u^{in}\leq b_N \right\} \quad \text{where} \quad b_{N}=\exp\left( -\frac{\alpha_2 (\eta+2\pi N)}{\omega_2}\right).
$$

For $i, j \in \NN$ large enough, let us consider the set $V_{ij}:=R_{0}(\mathcal{S}_{i})\cap\mathcal{S}_{j}$, where $R_0$ is the return map introduced in \S\ref{global transition}. To understand the shape of such intersection we first must notice that for $n\in \NN$ sufficiently large, we have 
$b_N < \varepsilon^{out}$  and:
\begin{lemma} 
\label{tec2}
Under conditions (P1)--(P4), there exists $N_0\in \NN$ such that $b_{N+1}>a_{N}$ for all $N>N_0$.
\end{lemma}

\begin{proof}
In this proof we make use of the fact that $\alpha_2<\alpha_1$ ($\Leftrightarrow \delta>1$).  Defining the sequence:
$$
\xi_N= \frac{\eta +2\pi(N+1)}{\eta +2\pi N} , \qquad N\in \NN,
$$
it is easy to see that  $(\xi_N)_N$ is decreasing and $\lim_{N\rightarrow +\infty} \xi_N = 1$. Therefore, there exists $N_0\in \NN$ such that:
$$
\forall N>N_0, \qquad  \frac{\eta +2\pi(N+1)}{\eta +2\pi N} < \frac{\alpha_1}{\alpha_2},
$$
which is equivalent to the existence of $N_0\in \NN$ such that:
$$ \forall N>N_0, \qquad \alpha_1(\eta + 2\pi N)> \alpha_2(\eta + 2\pi (N+1)).$$
Multiplying both sides of the prevous inequality by $-1$ and composing with the exponential map, we conclude that there exists $N_0\in \NN$ such that 
$$
\forall N>N_0, \qquad\exp \left(-\frac{\alpha_1}{\omega_2}(\eta + 2\pi N) \right)  < \exp \left(-\frac{\alpha_2}{\omega_2}(\eta + 2\pi (N+1)) \right),  
$$
completing the proof. 
\end{proof}
\bigskip

The geometrical interpretation of Lemma \ref{tec2} is stressed in Figure \ref{hyperchaos}. 
\bigbreak
Let $V_{ji}=R_{0}(S_{i})\cap S_{j}$. From the above estimations, for all $i,j\in \mathbb{N}$ large enough, we get $V_{ij} \neq \varnothing$. Each $V_{ij}$ consists of two connected components $V_{ij}^{k}$, with $k=1,2$ -- see Figure \ref{hyperchaos}. These sets are fully intersecting vertical slabs across $S_{j}$. Defining $H_{ij}^k=R_{0}^{-1}(V_{ij}^{k})$, the next lemma claims that $H_{ij}^k$ builds a fully intersecting horizontal slab across $S_{N}$.

\begin{figure}[ht]
\begin{center}
\includegraphics[height=12.0cm]{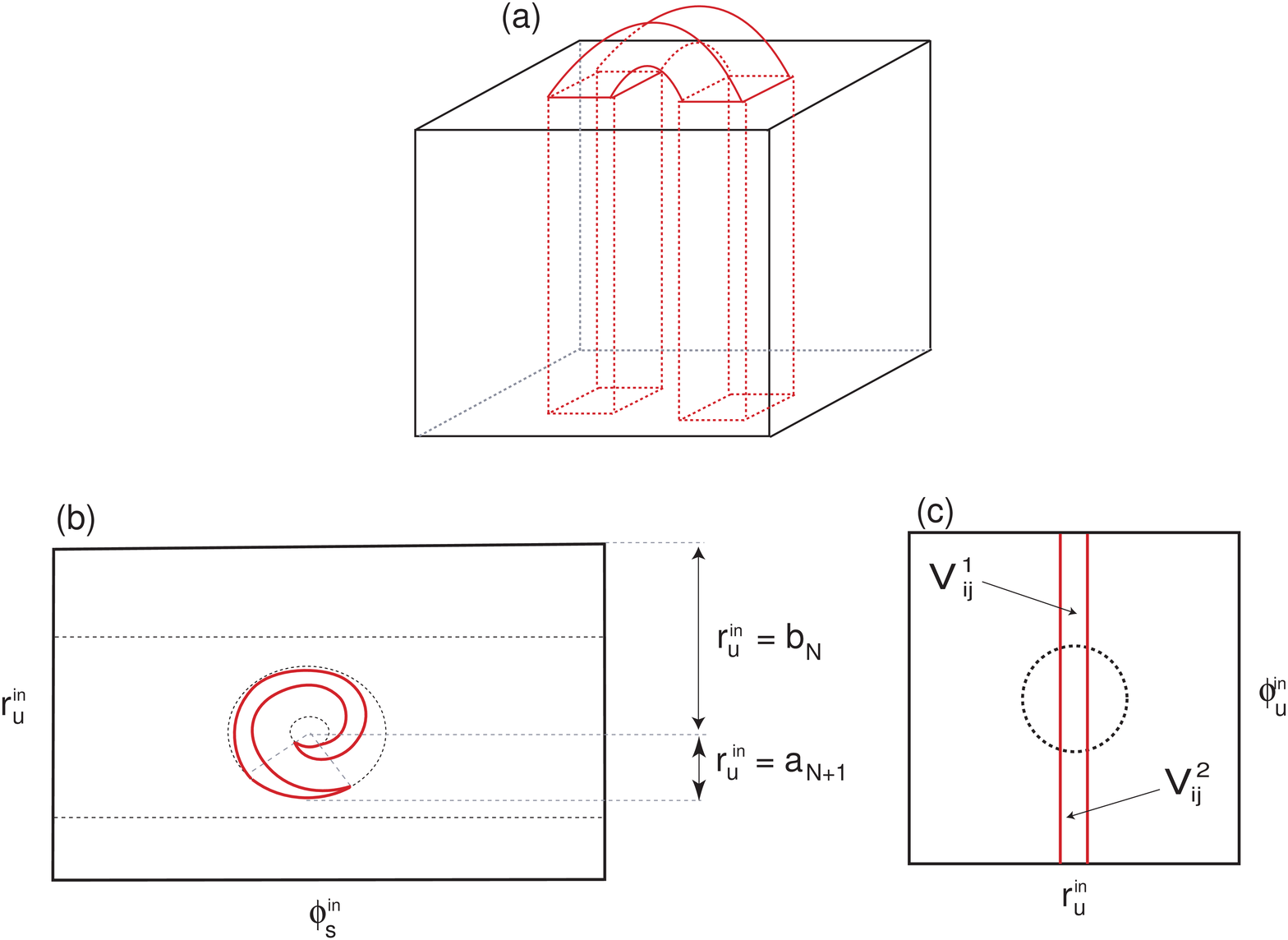}
\end{center}
\caption{\small The horseshoe with two contracting and one expanding directions. (a): global view, (b): upper view and (c): side view.}
\label{hyperchaos}
\end{figure}

\begin{lemma} [Slab Condition \cite{IbRo}]
\label{lemma12} For each $k=1,2$ and $i,j\in \mathbb{N}$ large enough, the set $H_{ij}^k$
is a fully intersecting horizontal slab in $S_{i}$.
\end{lemma}

By construction, the image under $R_{0}$ of the two horizontal boundaries of each $H_{ij}^{k}$ is a horizontal boundary of $V_{ij}^{k}$. Moreover, the image under $\Pi_O$ of each vertical boundary of $H_{ij}^{k}$ is a horizontal boundary of $\Pi_O(H_{ij}^{k})$ and then a vertical boundary of $R_0(H_{ij}^{k})$.
Now we need to obtain estimations of the rates of contraction and expansion of $H_{ij}^k$ under $R_0$ along the horizontal and vertical directions. The map $d$ is defined in the expression \eqref{dist}.

\begin{lemma} [Hyperbolicity condition \cite{IbRo}] 
\label{hyp_condition}
The following assertions hold:
\begin{itemize}
\item If $H$ is a $\mu_h$-horizontal slab intersecting $H_j$ fully, then $R_{0}^{-1}(H)\cap H_j=: \tilde{H}_j$ is a $\mu_h$-horizontal slab intersecting $H_j$ fully and $d(\tilde{H}_j)\leq \nu_h d(H)$, for some $\nu_h \in\, \, ]\, 0, 1\,[$.
\medbreak
\item If $V$ is a $\mu_v$-vertical slab contained in $\mathcal{S}_j$ such that $V\subset V_{ij}$, then $R_{0}(V)\cap \mathcal{S}_{j}$ is a  $\mu_v$-vertical slab contained in $\mathcal{S}_j$ and $d(R_{0}(V)\cap S_{j})<\nu_v d(V)$, for some $\nu_v \in\, \, ]\, 0, 1\,[$.
\end{itemize}
\end{lemma}

Disregarding, if necessary, a finite number of horizontal slabs (which is equivalent to shrink the initial cross section), we have:
\begin{proposition}
\label{shift_dyn}
There exists an $R_{0}$-invariant set of initial conditions $\Lambda_N \subset S \subset \Sigma_O^{in}$ on which the map $R_{0}$ is topologically conjugate to a full shift over a finite number of symbols. The maximal invariant set $\Lambda := \bigcup_{N\in \NN} \Lambda_N$ is a Cantor set. 
\end{proposition}


We will briefly review the proof to recollect the strategy for the construction of the Cantor set, which will be needed in the sequel.
\begin{proof}
Consider the set of points that remain in $\mathcal{E}=\bigcup_{N\in\NN}\mathcal{S}_N$ (see \eqref{sn}) under all backward and forward iterations of $R_{0}$. It may be encoded by a sequence of integers greater or equal to $n_0\in \NN$ as follows. Given a sequence of integers greater or equal to $N_0$, say $(s_N)_{N \in \ZZ}$, define
$$\Lambda^{-\infty}_{(s_N)_{N \in \ZZ}} = \left\{p \in \mathcal{E}: \left(R_{0}\right)^{-i}(p) \in \mathcal{S}_{s_{-i}},\,\,\forall \,i\,\in \NN\cup\{0\}\right\}$$
and
$$\Lambda^{-\infty} = \bigcup_{(s_N)_{N \in \ZZ}} \,\Lambda^{-\infty}_{(s_N)_{N \in \ZZ}}.$$
More generally, using the Slab Condition stated in Lemma \ref{lemma12}, for $k \in \NN$, set
\begin{eqnarray*}
\Lambda^{-1} &=& \bigcup_{(s_N)_{N \in \ZZ}} \,\left(R_{0}(V_{s_{-1}} )\cap \mathcal{S}_{s_0}\right) \equiv \bigcup_{(s_N)_{N \in \ZZ}} \,V_{s_0 s_{-1}}\\
\Lambda^{-2} &=& \bigcup_{(s_N)_{N \in \ZZ}} \, \left(R_{0}(V_{s_{-1}s_{-2}}) \cap \mathcal{S}_{s_0}\right) \\ & =&  \bigcup_{(s_N)_{N \in \ZZ}}\, \left(R_{0}^2(\mathcal{S}_{s_{-2}}) \cap R_{0}(\mathcal{S}_{s_{-1}}) \cap \mathcal{S}_{s_0}\right) \equiv \bigcup_{(s_N)_{N \in \ZZ}}\, V_{s_0 s_{-1}s_{-2}}\\
&\vdots & \\
\Lambda^{-k} &=& \bigcup_{(s_N)_{N \in \ZZ}} \,\left(R_{0}(V_{s_{-1}s_{-2}\ldots s_{-k}})  \cap \mathcal{S}_{s_0}\right) \\
& = &  \bigcup_{(s_N)_{N \in \ZZ}} \,(R_{0}^k (\mathcal{S}_{s_{-k}})\cap R_{0}^{k-1}(\mathcal{S}_{s_{-k+1}})\cap\ldots \cap R_{0}(\mathcal{S}_{s_{-1}}) \cap \mathcal{S}_{s_0}) \\
&\equiv& \bigcup_{(s_N)_{N \in \ZZ}} V_{s_0 s_{-1}\ldots s_{-k}}.
\end{eqnarray*}

Observe that, due to the Hyperbolicity Condition established in Lemma \ref{hyp_condition}, for every $k \in \NN\cup\{0\}$ the set $\Lambda^{-k}$ is a disjoint union of vertical slabs $V_{s_0 s_{-1}\ldots s_{-k}}$ contained in $\mathcal{S}_{s_0}$, where the width of $V_{s_0 s_{-1}s_{-2}\ldots s_{-k}}$ is $(\nu_v)^{k-1}$ times smaller than the width of $V_{s_0 s_{-1}}$ and, moreover, $\lim_{k \to +\infty}\, (\nu_v)^{k} = 0$ and $V_{s_0 s_{-1}s_{-2}\ldots s_{-k}}\subset V_{s_0 s_{-1}s_{-2}\ldots s_{-(k-1)}}$. Consequently, the set $\Lambda^{-\infty}$, which is $\bigcap_{k \in \NN}\,\Lambda^{-k}$, consists of infinitely many vertical slices whose boundaries lie on $\partial_v \mathcal{S}_{s_0}$ -- see Figure \ref{two spirals}. The construction of
$$\Lambda^{+\infty}=\bigcup_{(s_N)_{N \in \ZZ}} \,\left\{p \in \mathcal{E}: \left(R_{0}\right)^{i}(p) \in \mathcal{S}_{s_{-i}}, \,\forall \,i\,\in \NN\cup\{0\}\right\}$$
is similar. Finally, the set trapped in ${S}$ (see \eqref{def_S}) by all the iterations, forward and backward, of $R_{0}|_{S}$ is precisely
\begin{equation}\label{Lambda_def}
\Lambda = \Lambda^{-\infty} \cap \Lambda^{+\infty}
\end{equation}
and is the intersection of an uncountable set of vertical slices with an uncountable set of horizontal transverse slices.  
Therefore, by construction, the set $\Lambda = \bigcup_{N\in \NN} \Lambda_N:= \bigcup_{N\in \NN}( \Lambda \cap \mathcal{S}_N)$
is a Cantor set which is in a one-to-one correspondence with the family of bi-infinite sequences of a countable set of symbols (the itineraries of the $R_0$ orbits inside a partition defined by a family of disjoint vertical slabs) and where the dynamics of $R_{0}$ is conjugate to a Bernoulli shift with countably many symbols.
\end{proof}

\begin{figure}[ht]
\begin{center}
\includegraphics[height=4.9cm]{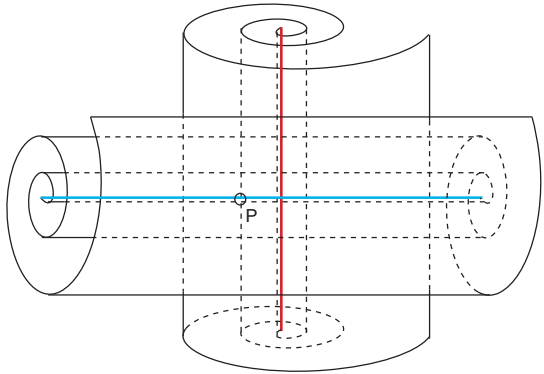}
\end{center}
\caption{\small Any point $P$ in $\Lambda$ is uniquely identified with two spiralling sheets. }
\label{two spirals}
\end{figure}

\begin{remark}[Whiskers]
\label{whiskers}
The previous construction allows us to conclude that when $\lambda=0$ and $N,M\in \NN$ large enough, the hyperbolic horseshoes $\Lambda_{N}$,  $\Lambda_{M}$ are heteroclinically related. More precisely, the unstable manifolds of the periodic orbits in $\Lambda_{N}$, are long enough to intersect the stable manifolds of the periodic points of $\Lambda_{M}$. That is, given two horizontal slabs, there exist periodic solutions jumping from one slab to another, and so the homoclinic classes associated to the infinitely many horseshoes are not disjoint. For $\lambda \approx 0$ small, this property persists for a finite (arbitrarily large) number  of horseshoes.  Regarding this subject, the chapter about the \emph{whiskers} of the horseshoes in Gonchenko \emph{et al} \cite{GST} is worthwhile reading.
\end{remark}


\subsection{Hyperbolicity of $\Lambda$:}
\label{hyperbolicity}
Using the expression \eqref{global_transition1}, in the local coordinates $(X, Y, Z)$, the eigenvalues of $DR_{0}$, when evaluated at points of $C^{in}\subset\Sigma_O^{in}$ with $X^2+Y^2 \gtrsim 0$ and  $Z\approx 0$ small enough, lie in different connected components of $\RR^2\backslash S^1$. According to \S 5.1 of \cite{FS},  the eigenvalues  $\mu_1, \mu_2$ and $\mu_3$ of $DR_0|_{(X_0, Y_0, Z_0)\in \Lambda }$  are real and satisfy:
\begin{equation}
\label{hyp_autovalores}
\mu_1= 
 -O\left(\sqrt{X_0^2+Y_0^2}\right)\approx 0, \qquad 1\gtrsim |\mu_2|\quad \text{and} \quad \mu_3  = 
O\left(\frac{-1}{\sqrt{X_0^2+Y_0^2}}\right) \ll -1.
\end{equation}
\medbreak
Their  eigendirections are such that:
\medbreak
\begin{itemize}
\item lines parallel to $W^s_{loc}(O)$  are contracted under $DR_0$ and
\medbreak
\item lines connecting $T_n^I$ and $T_n^O$ (see Figure \ref{annulus}) are stretched under $DR_{0}$.
\end{itemize}
\bigbreak
This agrees well with \S 3.2 of \cite{Wiggins}. These properties allow us to apply the construction of appropriate family of cones and so one concludes that  the map $R_0$ induces on the tangent bundle $T_\Lambda$ one expanding direction and two contracting directions -- \cite[\S 6]{KH95}. Item (c) of Theorem \ref{Th2b} is proved. 

\begin{remark}
For every compact invariant subset $C$ of the cross section $\Sigma\backslash W^s(O)$, the intersection $C \cap \Lambda$ is uniformly hyperbolic. 
\end{remark}

\begin{remark}
\label{complex_rem}
Numerically, using \emph{Maple}, it is possible to find open regions $\mathcal{W}$ in the parameter space $(X,Y,Z)$ for which the map $DR_0|_\mathcal{W}$ has a pair of complex (non-real) eigenvalues. The points of $C^{in}$ where $R_0$ is well defined and belong to $\mathcal{W}$
should be excluded from $\overline{\Lambda}$, although they may be homoclinically related to $\Lambda$. 
\end{remark}


\section{Generalized homoclinic tangency}

\label{perturbations}

The goal of this section is to prove that, in the $C^1$-topology, the map $R_0^{-1}$ may be approximated by a $C^r$ diffeomorphism with a homoclinic tangency satisfying \textbf{[T2]--[T4]} of Definition \ref{Tatjer_condition}, for $r\geq 5$. Once this is proved, by definition, the map $R_0$ may be approximated by a $C^r$ diffeomorphism with a Tatjer homoclinic tangency satisfying \textbf{[T1]--[T4]}  -- see page 257 of \cite{Tatjer}.  

\begin{remark}
We suggest the reader to think on the geometry of $R_0^{-1}$ as it was the first return map for a cycle to a bifocus at which the vector field has positive divergence ($\delta<1$). In particular, for $R_0^{-1}$, the local unstable manifold of $P_M$ may be seen as a two-dimensional disk crossing transversally $W_{loc}^s(O)\cap \Sigma_O^{in}$. Therefore, Proposition \ref{My_Prop5} may be applied to this disk.
\end{remark}

\subsection{First perturbation}
\label{first perturbation}

In this subsection, we prove the existence of a tangency associated to two periodic orbits $P_N, P_M$ of the invariant sets $\Lambda_N, \Lambda_M \subset \Lambda$, for some $N, M\in \NN$. 
\begin{definition}
Let $P_N$ be a periodic point of $\Lambda_N$ and $P_M$ a periodic point of $\Lambda_M$ of period arbitrarily large.
We say that two manifolds $W^u(P_N)$ and $W^s(P_M)$ have a \emph{quadratic tangency} (or contact of order 1) at $y_0$ is there exists an arc $\ell \subset W^s(P_M)$, a regular surface $S \subset W^u(P_N)$ and some $C^2$-change of coordinates on an open neighbourhood $U_{y_0}$ of $y_0$ such that:
\begin{itemize} 
\item $\dim W^u(P_N) = 2$ and $\dim W^s(P_M)=1$;
\item $y_0=(0,0,0)$;
\item $S=\{ (x,y,z) \in U_y: z=0\}$;
\item $\ell$ is a regular curve parametrized by $t$ as $\ell(t)=(x(t), y(t), z(t))$ and $\ell(0)=y_0$ and 
\item $z'(0)= 0$ and $z''(0) \neq 0$.
\end{itemize}

\end{definition}

\begin{figure}[ht]
\begin{center}
\includegraphics[height=2.3cm]{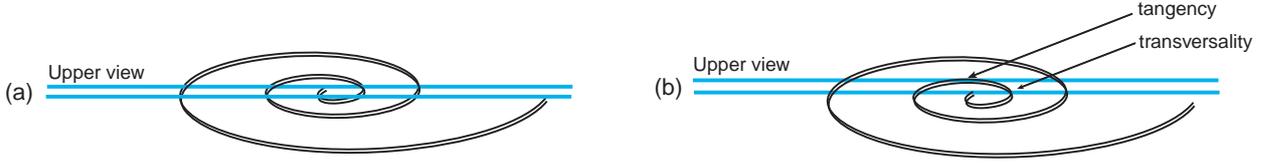}
\end{center}
\caption{\small  There are diffeomorphisms arbitrarily close to $R_{0}$ for which me may find a heteroclinic tangency associated to two hyperbolic periodic points of $\Lambda_N,\Lambda_M$, fome some $N,M\in \NN$, just by moving the saddle-value $\delta>1$. The intersection of spiralling sheets with a plane is locally a spiral (see Lemma 3 of \cite{Hart}), this is why the upper view has this spiralling shape. }
\label{Bif2}
\end{figure}

\begin{lemma}
\label{new configuration}
For $r\geq 5$, there exists a diffeomorphism $R_1$ defined on a subset of $\Sigma_O^{in}$, $C^r$-close to $R_0^{-1}$,   having two saddle periodic points, say $P_N$ and $P_M$, satisfying the following conditions:
\begin{enumerate}
\item $P_N$ and $P_M$ are arbitrarily close to each other;
\item the stability index of $P_N$ and $P_M$ is 1;
\item there is a quadratic tangency between $W^u(P_N)$ and $W^s(P_M)$;
\item $P_N$ and $P_M$ are heteroclinic related to each other, meaning that $W^u(P_M)  \pitchfork  W^s(P_N) \neq \emptyset$.
\end{enumerate}
\end{lemma}

\begin{proof}
For $\lambda=0$, the union of horseshoes $(\Lambda_N)_{N\in \NN}$ accumulates on $\Gamma$ as well as the invariant manifolds of its periodic orbits (see the proof of Proposition \ref{shift_dyn} and Remark \ref{whiskers}). 
The sets $W^u(R_0^{-1},P_N)$ and $W^u(R_0^{-1},P_M)$ may be seen as two-dimensional disks crossing transversally $W_{loc}^s(O)\cap \Sigma_O^{in}$. Therefore, using Proposition \ref{My_Prop5}, the sets $\Pi_O^{-1}\circ \Psi_1^{-1}(W_{loc}^u(R_0^{-1},P_N))$ and $\Pi_O^{-1}\circ \Psi_1^{-1}(W_{loc}^u(R_0^{-1},P_M))$ are spiralling sheets accumulating on $W^s_{loc}( O)\cap \Sigma_O^{in}$. Using Lemma 5.3 of \cite[pp. 435]{OS92}\footnote{The existence of sinks found in Lemma 5.3 of \cite{OS92} are a consequence of tangencies associated to dissipative periodic orbits. This remark has been pointed out by D. Turaev.}, we know that heteroclinic tangencies are dense in the family of vector fields satisfying (P1)--(P4), just by moving the saddle value $\delta$, as illustrated in Figure \ref{Bif2}. Thus, there exists a diffeomorphism $R_1$ defined on a subset of $\Sigma_O^{in}$, $C^r$-close to $R_0^{-1}$,   having two saddle periodic points in $\Lambda$, say $P_N \in \Lambda_N$ and $P_M \in \Lambda_M$, fome some $N,M\in \NN$, whose geometric configuration is as in the statement. 
\end{proof}

\bigbreak

\begin{remark}
\label{saddle-node}
Another possibility to get the geometric configuration of Lemma \ref{new configuration} is  generically breaking the cycle $\Gamma$ (using (P5)), making use of the theory described in \cite{FS, YA} for a family of vector fields. In Figure 18 of \cite{FS}, the authors pointed out the evolution of a fixed point (for the first return map) as function on its period. In particular, when $\lambda \rightarrow 0$, the corresponding period goes to $+\infty$. On the turning points of the snaking curve (see Figure \ref{snaking}), saddle-node and period doubling occur. These local bifurcations are the result of bigger global bifurcations associated to the unfolding of tangencies. 
 \end{remark}

\begin{figure}[ht]
\begin{center}
\includegraphics[height=6.1cm]{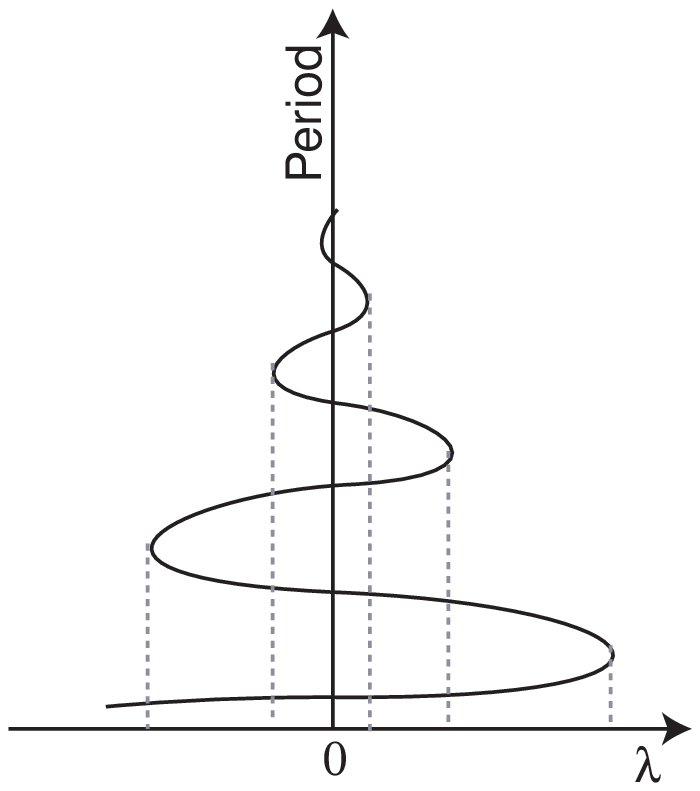}
\end{center}
\caption{\small  Snaking curve studied in \cite{FS} representing the evolution of the parameter $\lambda$ and the periodic point of the $n$-periodic points of the first return maps, $n\geq 1$. At the turning points of the snaking curve, the periodic point changes the stability, via a saddle-node and period doubling bifurcation.  }
\label{snaking}
\end{figure}

 \subsection{Second perturbation}
 \label{SS:TH}
 The main goal of the second perturbation is to obtain an equidimensional cycle  associated to two  heteroclinically related periodic points with different signatures. From now on, we make use of the following technical hypothesis (TH): 
 \medbreak
 \textbf{(TH):} For $r\geq 5$, there exists a  $C^r$-diffeomorphism ${R}_{2}$, $C^1$-close to $R_1$, with two saddle periodic points, say $P_N$ and $P_M$ in $C^{in}$, satisfying the conditions of Lemma \ref{new configuration} and such that the spectrum of $DR_2$ at $P_M$ has real eigenvalues satisfying \eqref{hyp_autovalores} and $DR_2$ at $P_N$ does not admit a dominated splitting into one-dimensional sub-bundles along its orbit.
 \medbreak
 The main focus of (TH) is not the co-existence of the two saddles $P_N$ and $P_M$ in $C^{in}$, but rather its  heteroclinic relation.  The lack of domination is a natural assumption due to the plethora of bifurcations  which arise either when we move $\delta>1$ (cf. Remark \ref{complex_rem}) or when the cycle is generically broken (cf. Remark \ref{saddle-node}).  
 
 \subsection{Third Perturbation}
 Using Franks' Lemma, the next result allows us to perform perturbations of the derivative in small neighbourhoods of the orbit of a point. Although the result holds in the $C^1$-topology, the resulting diffeomorphism may be $C^r$, with $r>1$ arbitrarily large.

\label{second perturbation}

\begin{lemma}
\label{new configuration2}
There exists a  $C^r$-diffeomorphism ${R}_{3}$, $C^1$-close to $R_2$, with two saddle periodic points, say $P_N$ and $P_M$, satisfying the following conditions (see Figure \ref{bif1apicture}):
\begin{enumerate}
\item $P_N$ and $P_M$ are arbitrarily close to each other;

\item the stability index of $P_N$ and $P_M$ is 1;
\item the expanding eigenvalues of $P_N$ and $P_M$ are non-real and real, respectively;
\item there is a quadratic tangency between $W^u(P_N)$ and $W^s(P_M)$;
\item $P_N$ and $P_M$ are heteroclinic related to each other, meaning that $W^u(P_M)  \pitchfork W^s(P_N) \neq \emptyset$.
\end{enumerate}
\end{lemma}

\begin{proof}
Let us start with the configuration given in Lemma \ref{new configuration} combined with (TH). Using Remark \ref{whiskers}, the period of $P_N$, say $\xi$, may be chosen arbitrarily large.  Since $DR_2$ has two real eigenvalues larger than 1 with no nilpotent part, then, for any $\varepsilon>0$, there is a neighbourhood $U$ of the periodic orbit of $P_N$ and a $\varepsilon$-perturbation $R_3$ of $R_2$ in the $C^1$-topology, such that:
\begin{itemize}
\item $R_3$ coincidies with $R_{2}$ outside $U$ and on the orbit of $P_N$;
\item the differential $DR_3^{\xi}|_{E^u}$ has a pair of complex eigenvalues with real part greater than 1, say $\mu^\star+\omega^\star i$.
\end{itemize}
This can be obtained by the Franks' Lemma \cite{Franks}. The derivative of the new diffeomorphism can be written as 
$$
DR_3= R_\theta \circ DR_2,
$$
where $R_\theta $ represents the $\theta$-rotation on the eigenplane $E^u$ associated to the eigenvalues $\mu_1^{-1}$ and $\mu_2^{-1}$. $DR_3$ should be the identity on $E^s$, the eigendirection associated to the eigenvalue $\mu_3^{-1}$. 
This perturbation is specific of the $C^1$-topology and is possible due to two important properties:
\begin{itemize}
\item within a compact set of $C^{in}\subset \Sigma_O^{in}$ containing $\Lambda_M$ and disjoint from the stable manifold of $P_M$, the norm of the eigenvalues $|\mu_1|^{-1}$ and $|\mu_2|^{-1}$ is uniformly bounded (and close to each other) and
\item the period of $P_N\in \Lambda_M$ may be taken large enough. If the period  of $P_N$ is not large enough,  in Lemma \ref{new configuration} we should choose another periodic point with larger period. 

\end{itemize}
By construction, the perturbation does not destroy the tangency described in Lemma \ref{new configuration}; in particular the configuration of Lemma \ref{new configuration} persists.
\end{proof}

Let  $P_N= \left(X,Y, Z\right)$ be a periodic point of $\Lambda_N$ as in Lemma \ref{new configuration2}.  Using the notation of the previous proof, in a small neighbourhood of $P_N $, we may define a local chart $(\overline{X}, \overline{Y}, \overline{Z})$ such that:
\begin{equation}
\label{new coordinates}
R_{3} (\overline{X}, \overline{Y}, \overline{Z})  = 
\left(\begin{array}{ccc}
\mu^\star \cos \left(2\pi \omega^\star\right) & -\mu^\star \sin \left(2\pi \omega^\star\right) & 0 \\ \\
\mu^\star \sin \left(2\pi  \omega^\star\right)& \mu^\star \cos \left(2\pi  \omega^\star\right)& 0 \\ \\
0& 0 & \mu_3^{-1}\\
\end{array}\right) \left(\begin{array}{lll}
\overline{X} \\ \\ \\
\overline{Y}\\ \\ \\
\overline{Z}\\
\end{array}\right)   
\end{equation}

\subsection{Fourth perturbation}
\label{irrational}
Concerning the periodic orbit $P_N$ of Lemma \ref{new configuration2}, if necessary,  we may locally perturb $R_3$ in such a way that $\omega^\star \in \RR\backslash \mathbb{Q}$. Let us denote by ${R}_4$ the resulting perturbation.

\subsection{Fifth perturbation}
\label{fourth perturbation}
In this section, we will use the strategy used by \cite{KNS}, adapted to our purposes. 

\begin{figure}[ht]
\begin{center}
\includegraphics[height=8.0cm]{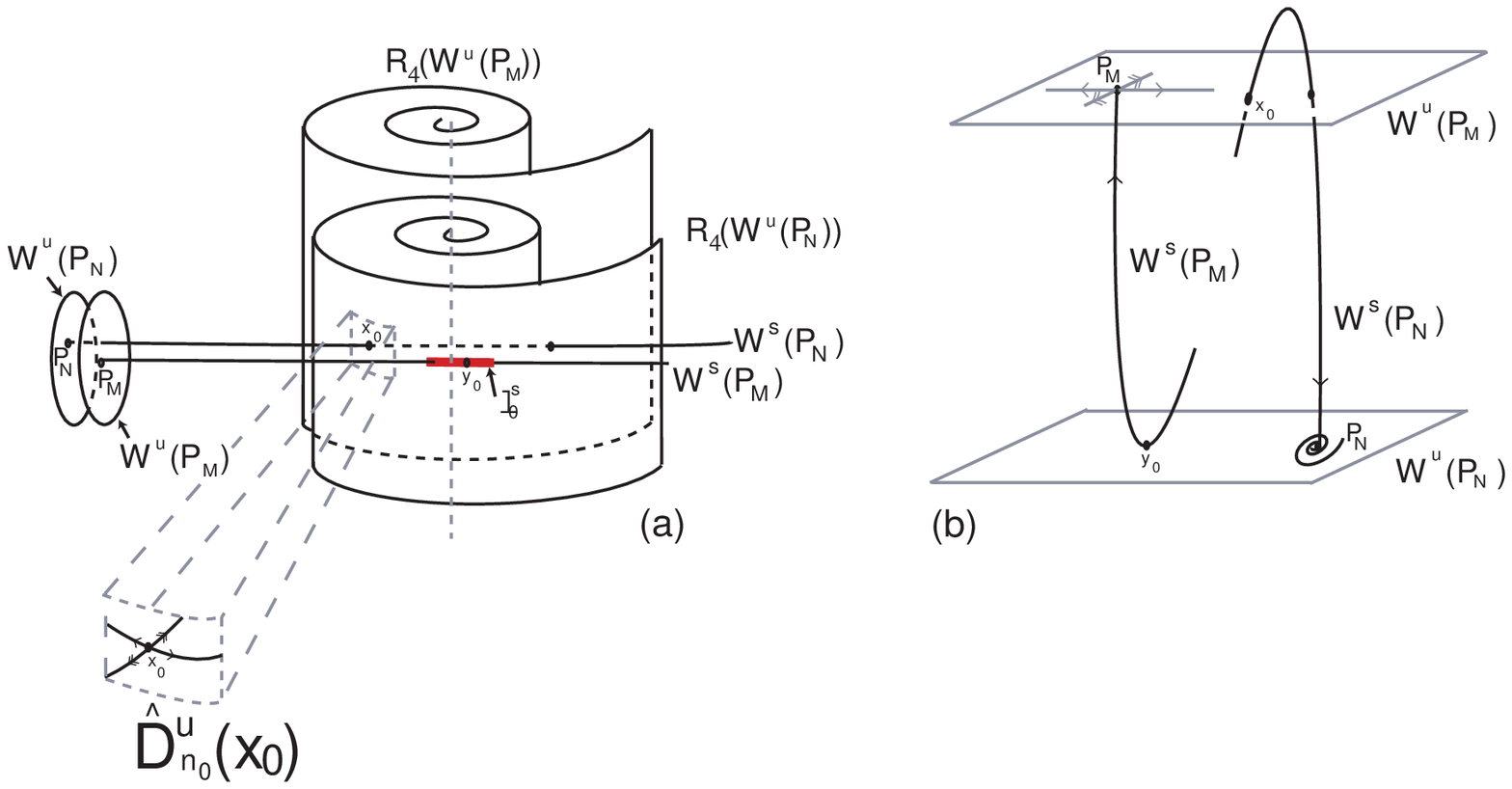}
\end{center}
\caption{\small Illustration of Lemma \ref{new configuration2}.}
\label{bif1apicture}
\end{figure}

\bigbreak
We start with the diffeomorphism $R_4$ given in Subsection \ref{irrational}. Let $x_0$ be a point in $R_4(W^u(P_M)) \cap W^s(P_N) \subset W^u(P_M) \cap W^s(P_N)$ (it exists by item (5) of Lemma \ref{new configuration2}). Suppose, without loss of generality that 
\begin{equation}
\label{Wuu(P_M)}
x_0 \notin W^{uu}(P_M).
\end{equation}
For each  $n\in \NN$, let us define the following sets (see Figure \ref{bif1apicture}):
\begin{itemize}
\medbreak
\item $D_0^u(x_0)$ is a small two-dimensional disk contained in $R_4(W^u(P_M))$ and containing $x_0$; 
\medbreak
\item  $x_n= {R}_4^n(x_0)$;
\medbreak
\item $\hat{D}_n^u=R_4^n(D_0^u(x_0)) \supset\{x_n\}$;
\medbreak
\item  $\ell^{uu}(x_0)$ is the segment contained in the leaf through $x_0$ of $\mathcal{F}^{uu}(P_M)$; 
\medbreak
\item   $\hat{\ell}_n^{uu}= R_4^n(\ell^{uu}(x_0))$ and
\medbreak
\item  $v_n^{uu}$ is the unitary  tangent vector to $\hat{\ell}_n^{uu}$ at $x_n$.

\end{itemize}

\bigbreak
By $\lambda$-Lemma \cite{PM}, it is straightforward that, for a large $n\in \NN$, there is a subset $D_1$ of $D_0^u(x_0)\subset W^u(P_M)$ containing $x_0$ such that  $R_3^n(D_1)\subset \hat{D}_n^u$ and converges to $W^u(P_N)$, in the $C^1$ topology, as $n \rightarrow +\infty$. 
Since $T \, W^u(P_N)$ is compact, there is $v_\infty \in T \, W^u(P_N)$ such that 
$
\lim_{n_i\in \NN} v_{n_i}^{uu} = v_\infty.
$
Denoting by $\| \,.\, \|$ the usual norm in the vector space $\RR^3$, and combining the two previous conclusions, we have proved that:
\begin{lemma}
\label{7.4}
For any $\varepsilon>0$, there is an integer $n_0 \in \NN$  and a sequence $(n_i)_i$ such that 
$$
\forall n_i>n_0, \quad |P_N - x_{n_i}|< \varepsilon \qquad \text{and} \qquad \|v_\infty- v_{n_i}^{uu}\|<\varepsilon.
$$
\end{lemma}

Now, let us consider a quadratic tangency $y_0$ between $W^u (P_N)$ and $W^s(P_M)$  (it exists by item (4) of Lemma \ref{new configuration2}). For every integer $m\in \NN_0$, define:
\begin{itemize}
\medbreak
\item $y_{-m}= R_4^{-m}(y_0)$;
\medbreak
\item $\ell_{-m}^s$ is a small arc of $W^s(P_M)$ passing through $y_{-m}$ (it exists by definition);
\medbreak
\item $w^s_{-m}$ is the unitary  vector tangent to $\ell^s_{-m}$ at  $y_{-m}$.
\medbreak
\end{itemize}

Therefore:
\medbreak
\begin{lemma}
\label{7.5}

For any $\varepsilon>0$, there is an integer $m_0 \in \NN$ and a subsequence $(m_i)_{n_i}$ such that
$$
\forall m_i \geq m_0, \quad |P_N - y_{-m_i}|< \varepsilon \qquad \text{and} \qquad \|v_\infty- w_{-m_i}^{s}\|<\varepsilon.
$$
\end{lemma}

\begin{proof}

 Since $\left|  \mu^\star \right | > 1$ (see formula \eqref{new coordinates}), ${y_{-m}}$ converges to $P_N$ as $m \rightarrow +\infty$ and the sequence $(\|w_{-m_i}^{s}\|)_i$ does not vanish.
Since  $\omega^\star\in \RR\backslash \mathbb{Q}$ (see 4th perturbation on \S \ref{irrational}), there exists a subsequence $(m_i)_{n_i}$ such that ${w_{-m_i}^s }$ converges to $v_\infty$ in $\RR^3$. 

\end{proof}

Using triangular inequality, combining Lemmas \ref{7.4} and \ref{7.5}, for any $\varepsilon>0$, there exist $n_0, m_0 \in \NN$ large enough, such that:
\begin{equation}
\label{triangular}
|x_{n_0} - y_{-m_0}|< 2\varepsilon \qquad \text{and} \qquad \|v_{n_0}^{uu}- w_{-m_0}^{s}\| \leq  \|v_\infty- v_{n_0}^{uu}\| + \|v_\infty- w_{-m_0}^{s}\|<2\varepsilon.
\end{equation}
\bigbreak
These two inequalities  are the key for the last perturbation. 
\medbreak

\begin{proposition}
\label{prop_final}
For  $\varepsilon>0$ small, there exists a  $C^r$-diffeomorphism ${R}_{5}$, $C^1$-$2\varepsilon$-close to $R_4$, such that:
\medbreak
\begin{enumerate}
\item  $R_5$ coincidies with $R_4$ outside a  small neighbourhood of $y_0$ (tangency between $W^u (P_N,R_4)$ and $W^s(P_M, R_4)$) ;
\medbreak
\item the hyperbolic continuation of $P_M$ has a homoclinic tangency satisfying conditions \textbf{[T2]--[T4]}.
\end{enumerate}
\medbreak
\end{proposition}
\begin{proof}
Based on inequalities \eqref{triangular}, the perturbation will be performed in two steps.
\medbreak

\textbf{First part:}
 the hyperbolic continuation $\ell^s_{-m_0} ({R}_5)$ is obtained from $\ell^s_{-m_0} ( R_4)$ by a shifting down operation along the stable axis tangent to the stable  manifold of $P_M$, as depicted in Figure \ref{shifting down}. Therefore the manifolds $W^u(P_M, R_5)$ and $W^s(P_M, R_5)$ have a quadratic tangency $z_0 \in \hat{D}^u_{n_0}$.  
 \medbreak
Let $\tilde{\ell}^{uu}_{n_0}$ be a curve in $\hat{D}^u_{n_0}$ passing through $z_0$ and such that $R_4^{-n_0} (\tilde{\ell}^{uu})$ is contained in one of leaves of $\mathcal{F}^{uu}(P_N)$. Generically, the space $T_{z_0} \ell^s_{-m_0}(\tilde{R}_4) $ does not coincide with $T_{z_0} \tilde{\ell}^{uu}$ but the second expression of \eqref{triangular} means that they are sufficiently close to each other.

\medbreak
\textbf{Second part:}
Perform a second perturbation to $R_4$, on a small neighbourhood of $y_0$, where the set 
$ T_{z_0} \ell^s_{-m_0}(R_5)$ is obtained from $T_{z_0} \ell^s_{-m_0}(R_4)$ by a small rotation around the stable axis meeting $\hat{D}^u_{n_0}$ orthogonally at $z_0$.  
 It is easy to see that we obtained a $C^r$ diffeomorphism $R_5$ such that $T_{z_0} \ell^s_{-m_0}(R_5)= T_{z_0} \ell^{uu}_{n_0}(R_5).$

 \begin{figure}[ht]
\begin{center}
\includegraphics[height=4.9cm]{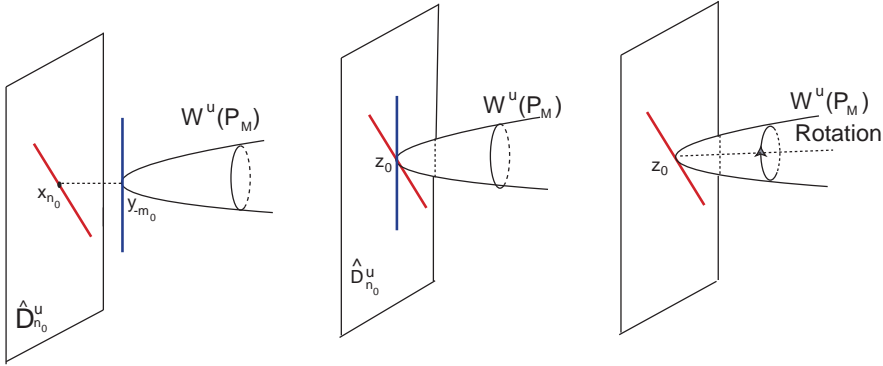}
\end{center}
\caption{\small Technique to obtain a homoclinic tangency of codimension two. Scheme of the shifting down (first part) and rotation (second part) performed by \cite{KNS}. }
\label{shifting down}
\end{figure}

Therefore, we have obtained the $C^r$-diffeomorphism  ${R}_5$, which is $C^1$-close to $R_4$, and such that:
\begin{itemize}

\item $r$ may be large enough;
\item $W^s(P_M)$ and $\hat{D}^u_{n_{0}}\subset W^u(P_M)$ have a quadratic tangency at $z_0$;
\item  the point $z_0$ does not belong to $W^{uu}(P_M)$  (see \eqref{Wuu(P_M)});
\item at the point $z_0$, the manifold $W^s(P_M, R_5 )$ is tangent to the leaf $\tilde{\ell}^{uu}_{n_0}(z_0)$ of $\mathcal{F}^{uu}(P_M)$ at $z_0$ and
\item the center stable bundle of $P_M$ and $W^u(P_M)$ are transverse at $z_0$.
\end{itemize}

This means that the hyperbolic continuation of $P_M$ has a homoclinic tangency satisfying conditions \textbf{[T2]--[T4]}, for $R_5$.
\end{proof}

\begin{remark}
\label{dissipativo}
Taking into account \eqref{hyp_autovalores}, the eigenvalues of $DR_5$ at $P_M$ are close to $\mu_1, \mu_2, \mu_3$ which satisfy $$|\mu_1|<|\mu_2| \lesssim 1<|\mu_3|$$ and
\begin{equation}
\label{dissipativo}
|\mu_1 \mu_3|\approx 0 <1, \qquad |\mu_2 \mu_3|\approx | \mu_3|>1 \qquad \text{and}\qquad |\mu_1\mu_2 \mu_3| \approx {\delta(X^2+Y^2)^{\delta-1}}  <1.
\end{equation}
This means that the periodic point $P_M$ is dissipative but not sectionally dissipative. 
\end{remark}

\begin{remark}
\label{R4in}
Using \eqref{dissipativo}, Proposition \ref{prop_final} allows us to conclude that the diffeomorphism $R_5^{-1}$ is $C^1$-close to $R_0$ and that $P_M$ is a periodic point satisfying \textbf{[T1]--[T4]}.
\end{remark}

\begin{remark}
\label{R4in}
By construction, although the diffeomorphism $R_5^{-1}$ is (just) $C^1$-close to $R_0$, it may be of class $C^r$, $r\geq 5$.
\end{remark}
\bigbreak
\subsection{Tatjer's conditions satisfied}
We present the next result which is important in the sequel: the existence of quasi-periodic behaviour and homoclinic tangencies associated to sectionally dissipative points. 
\begin{lemma}[Broer \emph{et al} \cite{Broer96}, Gonchenko \emph{et al}  \cite{GGT2007}, Tatjer \cite{Tatjer}, adapted]
\label{Tatjer_result}
Let $R$ be a $C^r$ $(r\geq 5)$ diffeomorphism on a three-dimensional smooth manifold which has a homoclinic tangency
to a dissipative periodic point $P$ whose map $DR(P)$ has real eigenvalues $\mu_1, \mu_2, \mu_3$ satisfying $|\mu_1|<|\mu_2|<1<|\mu_3|$, $|\mu_1\mu_3|<1$ and $|\mu_2\mu_3|>1$. In addition, suppose that the homoclinic tangency satisfies the Tatjer
conditions \textbf{[T1]--[T4]}. Then, there is a two-parameter family of diffeomorphisms $G_{(a,b)}$ with $G_{(0,0)}\equiv R$ such that:
\medbreak
\begin{enumerate}
\item for $n$ large enough, there are values of the parameter $(a_n, b_n)$, converging to $(0,0)$, for which the diffeomorphism $G_{(a_n,b_n)}$ undergoes a generic n-periodic Bogdanov-Takens bifurcation. 

\medbreak
\item there is a sequence $(a_n, b_n)_n$ of the parameter values converging to $(0, 0)$ such that, for any sufficiently large $n\in \NN$, $G_{(a_n,b_n)}$ has an $n$-periodic smooth attracting invariant circle;
\medbreak
\item there is a set $E$ of parameter values such that its intersection with any neighbourhood of $(0,0)$ (in the parameter space) has positive Lebesgue measure, and for $(a,b)\in E$ the diffeomorphism $G_{(a,b)}$ has a strange attractor near the orbit of tangency;
\medbreak
\item there are open sets $U \subset \RR^2$ arbitrarily close to $(0,0)$ such that for a generic $(a,b)\in U$, the diffeomorphism $G_{(a,b)}$ has infinitely many sinks.
\medbreak
\item arbitrarily $C^1$-close to any element  $G_{(a,b)}$, there there is a diffeomorphism $\tilde{G}$, not necessarily in the family $G_{(a,b)}$, exhibiting a generic homoclinic quadratic tangency to a sectionally dissipative periodic orbit. 
\end{enumerate}
\end{lemma}
\medbreak
\begin{remark}
The role played by the saddle-node bifurcations in the two-dimensional scenario (see \emph{e.g.} \cite{LR2015, YA}) will be played by the  Bogdanov-Takens bifurcation in the three-dimensional case \cite{Broer96}. For such bifurcation of periodic points, the corresponding spectrum has one unipotent eigenvalue (double eigenvalue equal to 1 with associated eigenspace of dimension 1). The attracting invariant circles are generated by the Horozov-Takens bifurcation for three-dimensional diffeomorphisms, corresponding to the points $B_n^{- -}$ in the proof of Theorem 4 of \cite{GGT2007}. See also Figure 7 of \cite{GGT2007}. 

\end{remark}

\begin{remark}
\label{attracting curves remark}
Proposition 4.1 of \cite{Tatjer} shows that, under hypotheses of Lemma \ref{Tatjer_result}, there exists a family of return maps $\tilde{F}_{(\tilde{a} ,\tilde{b})\in \RR^2}$ associated to the generalized homoclinic tangency, which may be written as:
\begin{equation}
\label{henon}
\tilde{F}_{(a,b)}(\tilde{x}, \tilde{y}, \tilde{z})=\left(\tilde{z}, \tilde{b}\tilde{z}, \tilde{a}+\tilde{y}+\tilde{z}^2 \right).
\end{equation}
The limit return map near the Bogdanov-Takens bifurcation is the conservative H\'enon map.

\end{remark}

\begin{remark}
The condition about the dissipativeness  $|\mu_1\mu_2\mu_3|<1$ of the period orbit is essential to compute the convergence of the coefficients $\tilde{a}, \tilde{b}$ in the limit map \eqref{henon} and cannot be relaxed. See Formulas (3.11), (3.14) and Remark 2 of \cite{GGT2007}. According to \cite{GGT2007},  if $\lim_{k \rightarrow +\infty} (\mu_1 \mu_2 \mu_3)^k \neq 0$, then:
$$
\lim_{k \in \NN} R_k= \lim_{k \in \NN}\left(\frac{2J_k}{M_2}+ O((\mu_1 \mu_2 \mu_3)^k)\right)
$$
 could be undefined. The precise definitions of $J_1$, $R_k$, $J_k$ and $M_2$ are described in expressions (2.10), (3.11), (3.14) and (4.1) of \cite{GGT2007}.
\end{remark}

\section{Proof of the main results}
\label{final_proof}
Using the previous sections, it is easy to check that the map $R_5^{-1}$ satisfies Lemma \ref{Tatjer_result} for the dissipative periodic point $P_M$. Therefore, the map $R_5^{-1}$ may be seen as the organizing center by which we can obtain strange attractors and non-trivial contracting wandering domains. In Lemma \ref{Tatjer_result}, the parameter $a$ is responsible for splitting the manifolds $W^s(P_M)$ and $W^u(P_M)$ and $b$ is the parameter unfolding the degeneracy related to condition \textbf{[T3]}.

\subsection{Strange attractors: sixth perturbation (A)}
\label{8.1}
By Lemma \ref{Tatjer_result},  there exists a diffeomorphism $G_A$ which is $C^1$-close to $R_5^{-1} \equiv G_{(0,0)}$ exhibiting a homoclinic quadratic tangency to a sectionally dissipative fixed point. Using \cite{Leal} and \cite{Viana93} revisited in Lemma \ref{Tatjer_result} and observing that $G_A$ is $C^1$-close to $R_0$, the statement of Theorem \ref{main_thB} follows.

\begin{figure}[ht]
\begin{center}
\includegraphics[height=6.0cm]{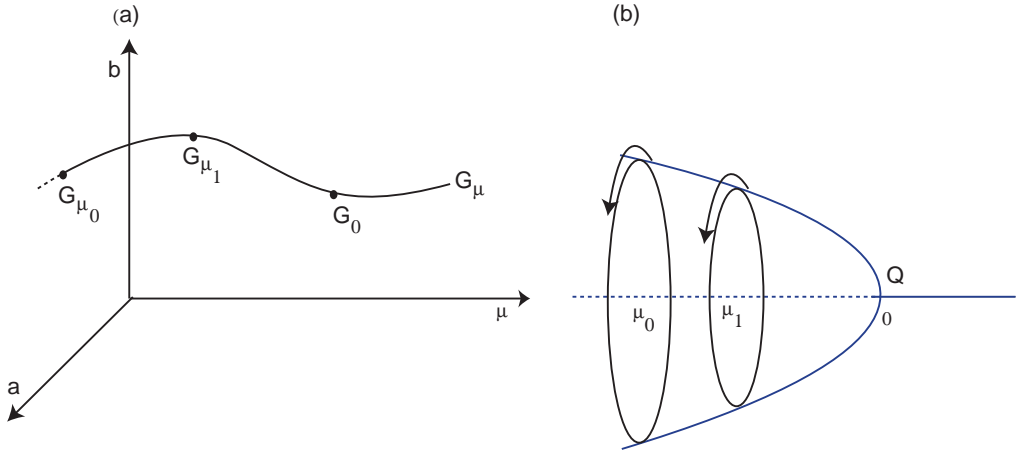}
\end{center}
\caption{\small Neimark-Sacker-Hopf bifurcation.  For $\mu=0$, the family $(G_{(a^\star,b^\star)},\mu)$ undergoes the generic Hopf bifurcation at a $n$-periodic point $Q$ for some integer $n \geq 1$.}
\label{hopf1}
\end{figure}

\subsection{Non-trivial wandering domains: sixth perturbation (B)}
\label{8.2}
 The last step of the proof of Theorem \ref{main_thA} follows the ideas of \cite{KNS}. For the sake of completeness, we revive the main steps of proof.
 \bigbreak
By Lemma \ref{Tatjer_result}(1), we know that $R_5^{-1} \equiv G_{(0,0)}$ is in the closure of the family of diffeomorphisms exhibiting differentiable attracting invariant curves. Let $(a^\star, b^\star)$ be a point in the parameter space such that $G_{(a^\star,b^\star)}$ exhibits a differentiable attracting invariant curve (see Theorem 4 of \cite{GGT2007}). 
 The authors of \cite{Broer96} proved that there exists a one-parameter family $\left(G_{(a^\star,b^\star)},\mu\right)_{\mu\in[-\epsilon, \epsilon]}$, $\epsilon >0$  small enough, of diffeomorphisms such that (Figure \ref{hopf1}):
 \begin{itemize}
 \item $\left(G_{(a^\star,b^\star)},\mu\right)$ is arbitrarily $C^5$-close to $(G_{(a^\star,b^\star)},0)$;
 \medbreak
 \item there exists $\mu_0 \in \, \, ]\,0, \epsilon\,[$ such that $\left(G_{(a^\star,b^\star)},{\mu_0}\right)\equiv G_{(a^\star,b^\star)}$;
 \medbreak
 \item for $\mu=0$, the family $(G_{(a^\star,b^\star)},\mu)$ undergoes the generic Hopf bifurcation at a $n$-periodic point, say  $Q=(0,0,0)$, for some integer $n \geq 1$.
 \end{itemize}
 
 \bigbreak
The Hopf bifurcation at $\mu=\mu_0$ creates the attracting invariant circle given in Lemma \ref{Tatjer_result}(2). In cylindrical coordinates $(r,\theta, t)$, in a small neighbourhood of $Q $, under generic conditions, we may write :
\begin{equation}
\label{normal form}
 \left(G^n_{(a^\star,b^\star)},\mu\right) (r, \theta \, mod{(2\pi)}, t)=\left((1+\mu)r-a_\mu r^3+O_\mu(r^4),\,\, \theta +\beta_\mu+O_\mu(r^2), \,\,\gamma t\right)
\end{equation}
where:
\medbreak
\begin{itemize}
\item $O_\mu (r^2)$ and $O_\mu (r^4)$ are smooth functions of order $r^2$ and $r^4$ near $(r, \mu) = (0, 0)$;
\medbreak
\item  $O_\mu (r^2)$ and $O_\mu (r^4)$ depend smoothly on $\mu$;
\medbreak
\item $a_\mu$, $\beta_\mu$ are real functions depending smoothly on $\mu$  with $a_0 > 0$ and
\medbreak
\item $\gamma\in \RR$ such that $0 < |\gamma |<1$. 
\end{itemize}
\bigbreak

Now we perform two perturbations to an element of the family $\left(G_{(a^\star,b^\star)},\mu\right)_{\mu\in[0, \mu_0]}$, say at $\mu=\mu_1$.

\textbf{First perturbation:}
We locally perturb $\left(G_{(a^\star,b^\star)},\mu_1\right)$ in such a way that the map in \eqref{normal form} satisfies the following conditions:
\begin{itemize}
\medbreak
\item $\frac{\beta_\mu}{ 2\pi }\in \RR\backslash \mathbb{Q}$ and
\medbreak
\item $O_{\mu_1} (r^2)=O_{\mu_1} (r^4)=0$.
\end{itemize}
\medbreak
 Let $\tilde{G}_1$ be the resulting diffeomorphism; it is clear that that there is an attracting invariant circle $\tilde{\mathcal{S}_1}$ and the restriction of $\tilde{G^n_1}$ to $\tilde{\mathcal{S}_1}$ is an irrational rotation. The radius of $\tilde{\mathcal{S}_1}$ is precisely $\sqrt{\frac{\mu_1}{a_{\mu_1}}}>0$. In particular, we have:
 \begin{equation}
 \label{wand1}
 \forall i \in \{0,..., n-1\}\qquad  \tilde{G}_1^i(\tilde{\mathcal{S}_1})\cap \tilde{\mathcal{S}_1}=\emptyset.
 \end{equation}
\bigbreak

\textbf{Second perturbation:} Following Denjoy's construction, let us construct a $C^1$ diffeomorphism $\tilde{G}_2$ arbitrarily $C^1$-close to $\tilde{G}_1$ (associated to a sequence $(I_i)_{i\geq 0}$ of open arcs, satisfying the properties described in \S \ref{Denjoy revision}, which is contained in a new circle $\tilde{\mathcal{S}_2}$ sufficiently $C^1$-close to $\tilde{\mathcal{S}_1}$) satisfying the following conditions:
\medbreak
\begin{itemize}
\item $\tilde{\mathcal{S}_2}$ is an attracting invariant circle for $\tilde{G}_2^n$;
\medbreak
\item for any $i,j>0$ with $i\neq j$:
\begin{equation}
\label{(4.3)}
\tilde{G}_2^n(I_i) = I_{i+1} \qquad \text{and} \qquad  I_i \cap I_j = \emptyset.
\end{equation}
\medbreak
\item $\omega(I_0, \tilde{G}_2^n)$ is a transitive Cantor set on $\tilde{\mathcal{S}_2}$ without periodic points.
\end{itemize}

\begin{figure}
\begin{center}
\includegraphics[height=6cm]{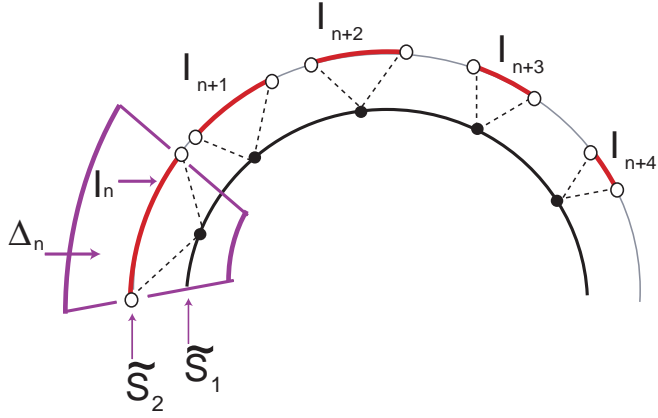}
\end{center}
\caption{\small Denjoy construction. For $i \in \{1,2\}$, the set  $\tilde{\mathcal{S}_i}$ is the attracting invariant and smooth curve for $\tilde{G^n_i}$. }
\label{denjoy}
\end{figure}
\bigbreak
\textbf{The contracting wandering domain:}
As illustrated in Figure \ref{denjoy}, consider a normal tubular neighbourhood of each arc $I_n$ which is defined as $
D_n =\bigcup_{x\in I_n} \Delta_n(x)
$
where $\Delta_n(x)$ is the open disk of radius $\delta_x>0$ centered at $x\in I_n$ lying in a plane normal to $I_n$ for each $n\geq 0$.

\begin{lemma}
The set $D_0$ is a contracting wandering domain for the diffeomorphism $\tilde{G}_2$.
\end{lemma}

\begin{proof}
Taking into account the way we constructed the two perturbations $\tilde{G}_1$ and $\tilde{G}_2$, it follows that 
$$\forall i,j \in \NN_0, \quad \tilde{G}_2^n(D_i) \subset D_{i+1} \qquad \text{and} \qquad D_i \cap D_j =\emptyset.$$
Since $\omega(I_0, \tilde{G}_2^n)$ is a transitive Cantor set, then $D_0$ is a wandering domain for $\tilde{G}_2^n$. The set $D_0$ is a contracting domain for $\tilde{G}_2^n$ because the first and third components of $\tilde{G}_2^n$ are contracting (see \eqref{normal form}).  Finally, taking into account \eqref{wand1} and \eqref{(4.3)}, we conclude that  $\tilde{G_2}^i(\mathcal{S}_2)\cap \mathcal{S}_{2} = \emptyset$ for $i=1, ..., n-1$, and therefore we get that $D_0$ is a contracting wandering domain for the diffeomorphism $\tilde{G}_2$.
\end{proof}

The map which satisfies Theorem \ref{main_thA} is $G_B:=\tilde{G}_2$ which, by construction, is $C^1$-close to $R_0$.

\section*{Acknowledgements} 
The author thanks the helpful and valuable comments from Pablo Barrientos, Artem Raibekas, Shin Kiriki and Teruhiko Soma about the main results of this paper.  The author was partially supported by CMUP (UID/MAT/00144/2013), which is funded by FCT with national (MEC) and European structural funds through the programs FEDER, under the partnership agreement PT2020. Alexandre Rodrigues also acknowledges financial support from Program INVESTIGADOR FCT (IF/00107/2015). Part of this work has been written during AR's stay in Nizhny Novgorod University, supported by the grant RNF 14-41-00044.


\newpage
\appendix
\section{Itinerary of the proof}
\label{app1}

\begin{table}[ht]
\begin{center}
\begin{tabular}{lclllcc}
Description & Notation& Property & Main Strategy & Section   \\

 & &  & of / based in &    \\ \hline
&&&& \\
Original map & $R_0$ & Infinitely many horseshoes &Fowler and Sparrow \cite{FS} &    \S \ref{suspended horseshoes}\\
& & accumulating on $\Gamma$ &  Ib\'a\~nez \emph{et al}  \cite{IbRo}  & \S \ref{hyperbolicity}    \\
& &  & Shilnikov \cite{Shilnikov67A, Shilnikov70}   &     \\
&&&& \\ \hline 
&&&& \\
$1st$ Perturbation & $R_1$    & Heteroclinic tangency & Ovsyannikov and &  \S \ref{first perturbation} \\ 
(within the family) &     & (orbits with the same signature) &   Shilnikov \cite{OS92}&    \\ &&&& \\ \hline
&&&& \\
$2nd$ Perturbation & $R_{2}$    & One orbit does not have& Technical & \S \ref{SS:TH} \\ 
 &     & dominated splitting  & Hypothesis  &    \\ 
  &     &  into 1-D sub-bundles  &   &    \\ &&&& \\\hline
  &&&& \\
 $3rd$ Perturbation & $R_{3}$    & Heteroclinic tangency & Franks \cite{Franks} & \S \ref{second perturbation} \\ 
 &     & (orbits with the different signatures) &  &    \\ &&&& \\\hline
 &&&& \\
$4th$ Perturbation 
& $R_4$  & Periodic point with complex  angle  &  &     \S \ref{irrational}\\ 
(if necessary)  &   & (in the Euler notation)  & &   \\ &&&& \\ \hline

 &&&& \\
 
$5th$ Perturbation & $R_{5}$    & Homoclinic tangency to a saddle  & Kiriki \emph{et al} \cite{KNS} &\S \ref{fourth perturbation}   \\
 &   & satisfying Tatjer conditions  &  &  \\ &&&& \\\hline

&&&& \\
$6th$ Perturbation  &$G_A$  & (H\'enon-like) strange attractors    & Leal \cite{Leal} and  &  \S \ref{8.1}\\ 
 Theorem A&  &    & Viana \cite{Viana93}  & \\ &&&&  \\\hline
&&&& \\
$6th$ Perturbation & $G_B$ & Contracting non-trivial     & Broer \cite{Broer96}, & \S \ref{8.2}\\
Theorem B&  & Wandering domains    & Denjoy \cite{Denjoy} & \\
&  &    &  Kiriki \emph{et al} \cite{KNS} & \\ 
&&&& \\ \hline
\end{tabular}
\end{center}
\bigbreak
\caption{Structure of the paper and itinerary of the proof of Theorems A and B.} 
\label{summary}
\end{table}

\end{document}